\pdfoutput=1
\documentclass[12pt]{article}
\usepackage[T1]{fontenc}
\usepackage[utf8]{inputenc}
\usepackage[nomath]{lmodern}
\usepackage[british]{babel}
\usepackage{amsmath}
\usepackage{amsthm}
\usepackage{amssymb}
\usepackage{geometry}
\usepackage{enumitem}
\usepackage{xcolor}
\usepackage{tikz-cd}
\usepackage{tikz}
	\usetikzlibrary{arrows.meta} 
	\usetikzlibrary{graphs} 
\usepackage[unicode, backref]{hyperref}

\newcommand*\cdef{\newcommand*}

\cdef \note [1]{%
	\textcolor{red}{[#1]}%
}

\cdef \maps {\colon} 
\cdef \cmp {\circ} 
\cdef \inv {^{-1}} 
\cdef \set [1] {\{#1\}}
\cdef \tuple [1] {\langle#1\rangle}
\cdef \card [1] {\lvert#1\rvert} 
\cdef \abs [1] {\lvert#1\rvert} 
\cdef \powset [1] {\mathcal{P}(#1)} 
\cdef \powfinset [1] {\mathcal{P}_{\omega}(#1)} 
\cdef \symdiff {\mathbin{\triangle}} 
\cdef \norm [1] {\lVert#1\rVert} 
\cdef \floor [1] {\left\lfloor #1\right\rfloor} 
\cdef \ceil [1] {\left\lceil #1\right\rceil} 
\cdef \K {\mathcal{K}} 
\cdef \Aut {\operatorname{Aut}} 
\cdef \Dist {\operatorname{Dist}} 
\cdef \Age {\operatorname{Age}} 
\cdef \im {\operatorname{im}} 
\cdef \id {\operatorname{id}} 
\cdef \RR {\mathbb{R}} 
\cdef \QQ {\mathbb{Q}} 
\cdef \ZZ {\mathbb{Z}} 
\cdef \NN {\mathbb{N}} 
\cdef \F {\mathcal{F}} 
\cdef \mapsfrom {\mathrel{\reflectbox{\ensuremath{\mapsto}}}}

\newcommand{\bis}{\operatorname{bs}} 

\theoremstyle{plain}
\newtheorem{theorem}{Theorem}[section]
\newtheorem{proposition}[theorem]{Proposition}
\newtheorem{corollary}[theorem]{Corollary}

\theoremstyle{definition}
\newtheorem{definition}[theorem]{Definition}
\newtheorem{observation}[theorem]{Observation}
\newtheorem{remark}[theorem]{Remark}

\newtheorem{example}[theorem]{Example}
\newtheorem{question}[theorem]{Question}
\newtheorem{construction}[theorem]{Construction}

\let \leq \leqslant 
\let \geq \geqslant

\setlist{itemsep = 0pt}
\setlist[enumerate]{leftmargin=*} 
\setlist[enumerate, 1]{label=\upshape (\arabic*), ref=(\arabic*)}
\setlist[enumerate, 2]{label=\upshape (\arabic{enumi}\alph*), ref=(\arabic{enumi}\alph*)}

\hypersetup{
	pdftitle = {Homogeneous isosceles-free spaces},
	pdfauthor = {Christian Bargetz, Adam Bartoš, Wiesław Kubiś, Franz Luggin},
	colorlinks,
	linkcolor = [rgb]{0, 0, 0.7},
	citecolor = [rgb]{0, 0.7, 0},
	urlcolor = [rgb]{0.4, 0.2, 0},
	linktoc = section,
}

\title{Homogeneous isosceles-free spaces}
\author{Christian Bargetz \\
		\small \href{mailto:christian.bargetz@uibk.ac.at}{\nolinkurl{christian.bargetz@uibk.ac.at}} \\
		\small Universität Innsbruck, \\
		\small Department of Mathematics, \\
		\small Technikerstraße 13, 6020 Innsbruck, Austria
	\and Adam Bartoš \\
		\small \href{mailto:bartos@math.cas.cz}{\nolinkurl{bartos@math.cas.cz}} \\
		\small Institute of Mathematics, \\
		\small Czech Academy of Sciences, \\
		\small Žitná 25, 115 67 Prague, Czech Republic
	\smallskip
	\and Wiesław Kubiś \\
		\small \href{mailto:kubis@math.cas.cz}{\nolinkurl{kubis@math.cas.cz}} \\
		\small Institute of Mathematics, \\
		\small Czech Academy of Sciences, \\
		\small Žitná 25, 115 67 Prague, Czech Republic
	\and Franz Luggin \\
		\small \href{mailto:Franz.Luggin@student.uibk.ac.at}{\nolinkurl{Franz.Luggin@student.uibk.ac.at}} \\
		\small Universität Innsbruck, \\
		\small Department of Mathematics, \\
		\small Technikerstraße 13, 6020 Innsbruck, Austria
}

\date{May 2024} 
\begin{document}

\maketitle

\begin{abstract}
	We study homogeneity aspects of metric spaces in which all triples of distinct points admit pairwise different distances; such spaces are called \emph{isosceles-free}.
	In particular, we characterize all homogeneous isosceles-free spaces up to isometry as vector spaces over the two-element field, endowed with an injective norm. Using isosceles-free decompositions, we provide bounds on the maximal number of distances in arbitrary homogeneous finite metric spaces.
	
	\medskip
	
	
\end{abstract}

\tableofcontents

\section{Introduction}

A mathematical structure is called \emph{ultrahomogeneous} if every isomorphism between its finite (or, more generally, finitely generated), substructures extends to an automorphism. Adding bounds on the cardinality of the substructures we obtain \emph{$n$-homogeneity}, where $n \geq 1$ is a natural number. Countable (or, more generally, countably generated) ultrahomogeneous structures are known in model theory as \emph{Fraïssé limits} (see e.g. Hodges~\cite{Hodges}) and they are fully characterized as unique countable ultrahomogeneous structures generated by a given class of finite (or finitely generated) structures satisfying some natural axioms, where the most important one is the \emph{amalgamation property}. Metric spaces can be easily viewed as first order structures, for instance, replacing the metric by countably many binary relations saying that ``the distance is less than a fixed positive rational number''. In this setting, isomorphisms are just bijective isometries and a metric space is ultrahomogeneous if every isometry between its finite subsets extends to a bijective auto-isometry of the space. Perhaps the first and arguably most important example is the \emph{Urysohn space}~\cite{Urysohn1927}, the unique separable complete ultrahomogeneous metric space $\mathbb U$ containing isometric copies of all separable metric spaces. The space $\mathbb U$ contains a dense countable ultrahomogeneous subspace in which all distances are rational, this is actually \emph{the} Fraïssé limit of the class of all finite metric spaces with rational distances.

In this paper we consider the special class of metric spaces without isosceles triangles, called \emph{isosceles-free}, in connection with homogeneity.
It turns out this class is a nice source of examples in the context of Fraïssé theory as well as in the context of finite combinatorics and the question of how many distinct distances a finite homogeneous spaces of a fixed size can have.

Our main results include:
\begin{enumerate}
\item \label{res:ultra-free}
  Realizing that every $1$-homogeneous isosceles-free metric space is already ultrahomogeneous (Proposition~\ref{ultra-free}), and that homogeneous isosceles-free spaces are exactly \emph{uniquely $2$-homogeneous} spaces (Proposition~\ref{unique-two-homog}).
\item \label{res:WAP_fails}
  Showing that the class of all finite isosceles-free metric spaces is a hereditary class without the \emph{weak amalgamation property} (Theorem~\ref{WAP_fails}).
\item \label{res:boolean}
  Characterizing all homogeneous isosceles-free spaces up to isometry as normed $\ZZ_2$-linear spaces with an injective norm (Theorem~\ref{thm:homog_iso-free}), using an auxiliary notion of a \emph{Boolean metric space}.
\item \label{res:decomp}
  Studying more general $1$-homogeneous spaces through the lens of \emph{singleton distances} (i.e. locally non-repeating distances) and related invariant decompositions, showing that every $1$-homogeneous metric space is Boolean or \emph{isosceles-generated} or a \emph{rainbow duplicate} of an isosceles-generated space (Theorem~\ref{thm:homogeneous_cases}).
  In the case of $2$-homogeneous spaces, this further reduces to being isosceles-generated or isosceles-free.
\item \label{res:bounds}
  Giving bounds on the maximal number of distances in a homogeneous finite metric space of size $n$.
  In the case of a $2$-homogeneous space, we have the optimal bound $2^m(k + 1)$ for $n = 2^m(2k + 1)$, realized even by ultrahomogeneous spaces (Theorem~\ref{thm:number-of-distances} and Example~\ref{bound_attained}).
  This bound is optimal also for $1$-homogeneous spaces whose size is odd or a power of two.
  In the case of an even-sized $1$-homogeneous space of size $2^m(4k + 2)$ we give a better lower bound $2^m(3k + 2)$ (Example~\ref{ex:rainbow-on-2-homogeneous-space}).
\end{enumerate}

The paper is organized as follows.
In Section~\ref{sec:homog} we gather various notions of homogeneity of metric spaces and prove general preservation theorems.
In Section~\ref{sec:iso-free} we study isosceles-free metric spaces in general and in connection with $1$-homogeneity. We prove the results \ref{res:ultra-free} and \ref{res:WAP_fails} as well as the fact that the automorphism group of an isosceles-free space is Boolean.
	We also give a couple of illustrative examples.

In Section~\ref{sec:boolean} we further exploit the fact that homogeneous isosceles-free spaces are uniquely $1$-homogeneous and have a Boolean automorphism group. We call metric spaces with the latter properties \emph{Boolean metric spaces} and prove that they admit a certain $\ZZ_2$-linear/affine structure.
This leads to the proof of the complete classification of homogeneous isosceles-free spaces \ref{res:boolean}.
Later in the section we give infinite Cantor-like examples of homogeneous isosceles-free spaces (Example~\ref{ex:Cantor2} and \ref{ex:Cantor3}), demonstrating that metric completion may break ultrahomogeneity and the property of being isosceles-free.

In Section~\ref{sec:decomp} we study invariant decompositions of homogeneous metric spaces based on singleton distances – the decomposition into isosceles-free components and the decomposition into isosceles-generated components – in order to prove \ref{res:decomp}.
We also introduce the construction of a \emph{rainbow duplicate} of a $1$-homogeneous metric space, and show that this particular construction in fact realizes all $1$-homogeneous spaces with two isosceles-generated components.

In Section~\ref{sec:distances} we exploit the structural properties and constructions of homogeneous metric spaces obtained in previous sections to give a partial answer to the question: how many distinct distances can a finite homogeneous metric space of a fixed size $n$ have?
We obtain the bounds~\ref{res:bounds}.
Concrete values of the bounds are summarized in Table~\ref{table}.

\medskip

Let $X$ be a metric space. We use the following notation.
\begin{itemize}
	\item The distance is usually denoted by $d(x, y)$. We sometimes use $d_X$ instead of $d$ for clarity.
	\item $\Dist(X)$ denotes the set of used distances $\set{d(x, y): x, y \in X}$.
	\item $\Aut(X)$ denotes the automorphism group of all isometries $X \to X$.
		Note that here the word \emph{isometry} stands for \emph{isometric isomorphism} and not \emph{isometric embedding}.
	\item $\Age(X)$ denotes the class of all finite metric spaces isometrically embeddable into $X$.
\end{itemize}

\section{Homogeneity} \label{sec:homog}

A metric space $X$ is said to be 
\begin{itemize}
	\item \emph{$n$-homogeneous} for $n \in \NN_+=\{1,2,3,\dotsc\}$ if for every isometry $f\maps A \to B$ between subspaces $A, B \subseteq X$ with $\card{A} \leq n$ there exists an automorphism $F\maps X \to X$ extending $f$, i.e. $F|_A = f$;
	\item \emph{ultrahomogeneous} if it is $n$-homogeneous for every $n \in \NN_+$;
	\item \emph{uniquely $n$-homogeneous} if for every isometry $f\maps A \to B$ for $A, B \subseteq X$ with $0 < \card{A} \leq n$ there exists a unique $F \in \Aut(X)$ with $F|_A = f$;
	\item \emph{uniquely ultrahomogeneous} if it is uniquely $n$-homogeneous for every $n \in \NN_+$.
\end{itemize}
Note that $X$ is uniquely $n$-homogeneous if and only if it is $n$-homogeneous and uniquely $1$-homogeneous. In model theory, uniquely $1$-homogeneous structures are called \emph{Ohkuma structures}, see~\cite{GiraudetHolland} and~\cite{Ohkuma}.
We usually avoid the term \emph{homogeneous} as it can mean either $1$-homogeneous or ultrahomogeneous in the literature.

\begin{definition}
	Let $X$ be a metric space.
	We say that a subspace $Y \subseteq X$ is \emph{quasi-invariant} if for every $f \in \Aut(X)$ such that $f[Y] \cap Y \neq \emptyset$ we have $f[Y] = Y$.
\end{definition}

\begin{example}
	The metric space $X=\set{\tuple{i,j}: 1\leq i\leq 4,1\leq j\leq 2}$ with the metric \begin{align*}
		d(\tuple{i_1,j_1},\tuple{i_2,j_2})=\begin{cases}
			2&j_1\neq j_2,\\
			1&i_1\neq i_2,\\
			0&\text{else.}
		\end{cases}
	\end{align*}
	has two quasi-invariant subspaces $Y_j=\set{\tuple{i,j}: 1\leq i\leq 4}$ for $j=1,2$ (see Figure \ref{fig:quasi-invariant}).
	
	\begin{figure}[!ht]
		\centering
		
		\begin{tikzpicture}[every node/.style={draw,shape=circle,fill,minimum size=4pt,inner sep=0pt,outer sep=1pt}]
			\definecolor{c-d1}{RGB}{255,0,0}
			\definecolor{c-d2}{RGB}{0,255,0}
			\definecolor{c-d3}{RGB}{0,0,255}
			
			\node (v1) at (0, 0) {};
			\node (v2) at (2, 0) {};
			\node (v3) at (0, 2) {};
			\node (v4) at (2, 2) {};
			
			\node (v5) at (5, 1) {};
			\node (v6) at (7, 1) {};
			\node (v7) at (5, 3) {};
			\node (v8) at (7, 3) {};

			\draw[thick] (v1)--(v2)--(v3)--(v4)--(v1)--(v3);
			\draw[thick] (v2)--(v4);
			\draw[thick] (v5)--(v6)--(v7)--(v8)--(v5)--(v7);
			\draw[thick] (v6)--(v8);
			\node[draw=none, fill=none] at (6,0.5) {$Y_2$};
			\node[draw=none, fill=none] at (1,2.5) {$Y_1$};

		\end{tikzpicture}
		
		\caption{The metric space $X=Y_1\cup Y_2$ with all edges of distance $1$ drawn. All pairs of distinct points without an edge between them have distance $2$.}
		\label{fig:quasi-invariant}
	\end{figure}
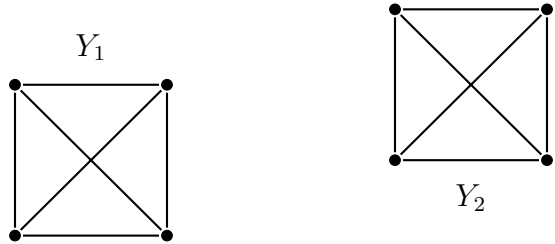
	
	To check that $Y_1$ is quasi-invariant, note that if $f$ is an automorphism and maps any point $x$ in $Y_1$ to $Y_1$, then due to $f$ being an isometry, we have $f[Y_1]=Y_1$ since $Y_1$ is exactly the set of points of distance $\leq 1$ to $x$. This remains true regardless of which distances (or how many distinct ones) we choose between a point in $Y_1$ and a point in $Y_2$, as long as no such distance is chosen as $1$.
\end{example}

More generally, any subspace $Y \subseteq X$ such that $d[Y \times Y] \cap d[(X \setminus Y) \times Y] = \emptyset$ or any component of an \emph{invariant decomposition} of $X$ (see Definition~\ref{def:invariant_decomposition}) is quasi-invariant.

\begin{proposition} \label{thm:ultrahom_subspace}
	Let $X$ be a metric space and let $Y \subseteq X$ be a quasi-invariant subspace.
	If $X$ is (uniquely) $n$-homogeneous for some $n \in \NN_+$ or ultrahomogeneous, then so is $Y$, and this is witnessed by restrictions of automorphisms of $X$.
	
	\begin{proof}
		Let $f\maps A \to B$ be an isometry between nonempty finite subspaces $A, B \subseteq Y \subseteq X$.
		If $X$ is $\card{A}$-homogeneous, there is $F \in \Aut(X)$ extending $f$.
		We have $\emptyset \neq B \subseteq Y \cap F[Y]$, and hence $F|_Y \in \Aut(Y)$ since $Y$ is quasi-invariant.
		Moreover, if $F$ is the unique automorphism of $X$ extending $f$, then $F|_Y$ is the unique automorphism of $Y$ extending $f$.
	\end{proof}
\end{proposition}

For metric spaces $X$ and $Y$ let $X \times_1 Y$ denote the product space endowed with the $\ell_1$-metric: $d(\tuple{x_1, y_1}, \tuple{x_2, y_2}) = d_X(x_1, x_2) + d_Y(y_1, y_2)$.
Also for every $f \in \Aut(X)$ and $g \in \Aut(Y)$ let $f \times g$ denote the map $\tuple{x, y} \mapsto \tuple{f(x), g(y)}$.

\begin{proposition} \label{thm:homogeneous_product}
	Suppose that $X$ and $Y$ are nonempty metric spaces such that the map $+\maps \Dist(X) \times \Dist(Y) \to \Dist(X \times_1 Y) \subseteq [0, \infty)$ is injective (and so bijective).
	Then $\tuple{f, g} \mapsto f \times g$ is a group isomorphism $\Aut(X) \times \Aut(Y) \to \Aut(X \times_1 Y)$.
	Moreover, $X \times_1 Y$ is (uniquely) $n$-homogeneous/ultrahomogeneous if and only if $X$ and $Y$ are.
\end{proposition}
\begin{proof}
	Distances in $X \times_1 Y$ are of the form $d(\tuple{x_1,y_1},\tuple{x_2,y_2})=d_X(x_1,x_2)+d_Y(y_1,y_2)$. Therefore, since sums of the form $d_X(x_1,x_2)+d_Y(y_1,y_2)$ injectively map into $[0,\infty)$, there will be a one-to-one correspondence between $\Dist(X \times_1 Y)$ and $\Dist(X)\times\Dist(Y)$.
	
	For every $f, f' \in \Aut(X)$ and $g, g' \in \Aut(Y)$ we have
	\begin{align*}
		d((f \times g)(x_1, y_1), (f \times g)(x_2, y_2)) 
			&= d_X(f(x_1), f(x_2)) + d_Y(g(y_1), g(y_2)) \\
			&= d_X(x_1, x_2) + d_Y(y_1, y_2)
			= d(\tuple{x_1, y_1}, \tuple{x_2, y_2}),
	\end{align*}
	and $(f \times g) \circ (f' \times g') = (f \circ f') \times (g \circ g')$.
	Hence, $f \times g \in \Aut(X \times_1 Y)$ and $\tuple{f, g} \mapsto f \times g$ is a group homomorphism $\Aut(X) \times \Aut(Y) \to \Aut(X \times_1 Y)$.
	
	Let $\pi_X\maps X \times_1 Y \to X$ and $\pi_Y\maps X \times_1 Y \to Y$ denote the projections.
	For every $f \in \Aut(X)$, $g \in \Aut(Y)$, $x \in X$, and $y \in Y$ we have $\pi_X((f \times g)(x, y)) = f(x)$ and $\pi_Y((f \times g)(x, y)) = g(y)$, and hence the homomorphism $\tuple{f, g} \mapsto f \times g$ is injective.
	
	To show that it is also surjective and to show the remaining claims, let $\phi\maps A \to B$ be an isometry of some subspaces $A, B \subseteq X \times_1 Y$.
	We prove that $\phi = (\phi_X \times \phi_Y)|_A$ for some isometries $\phi_X\maps \pi_X[A] \to \pi_X[B]$ and $\phi_Y\maps \pi_Y[A] \to \pi_Y[B]$.
	To that end, let us look at
	\[
		d(\phi(x_1, y_1), \phi(x_2, y_2)) = d(\tuple{x_1, y_1}, \tuple{x_2, y_2}) = d_X(x_1, x_2) + d_Y(y_1, y_2)
	\]
	and note that if $\phi(x_1,y_1)=:\tuple{a,b}$ and $\phi(x_2,y_2)=:\tuple{c,d}$, then
	\[
		d_X(x_1,x_2)+d_Y(y_1,y_2)=d_{X \times_1 Y}(\tuple{a,b},\tuple{c,d})=d_X(a,c)+d_Y(b,d)
	\]
	and it follows from our injectivity assumption of $+$ on $\Dist(X)\times\Dist(Y)$ that $d_X(x_1,x_2)=d_X(a,c)$ and $d_Y(y_1,y_2)=d_Y(b,d)$.
	In particular, if $x_1=x_2$, then $a=c$, so for all pairs of points $\tuple{x,y_1},\tuple{x,y_2}$ with identical $X$-components, we get that the $X$-components of the images under $\phi$ coincide as well: $\pi_X(\phi(x,y_1)) = \pi_X(\phi(x,y_2)) =: \phi_X(x)$.
	Also, $d_X(x_1, x_2) = d_X(a, c)$ shows that $\phi_X$ is an isometry.
	Analogously, we obtain the isometry $\phi_Y$.
	
	Hence, $\times\maps \Aut(X) \times \Aut(Y) \to \Aut(X \times_1 Y)$ is surjective.
	It also follows that every subspace $X \times \set{y} \subseteq X \times_1 Y$ (which is isometric to $X$) is quasi-invariant, and so if $X \times_1 Y$ is (uniquely) $n$-homogeneous/ultrahomogeneous, so is $X$ by Proposition~\ref{thm:ultrahom_subspace}, and similarly for $Y$.
	Finally, if $\phi_X$ and $\phi_Y$ have (unique) extensions $\Phi_X \in \Aut(X)$ and $\Phi_Y \in \Aut(Y)$, then $\Phi_X \times \Phi_Y$ is a (unique) extension of $\phi$, so if $X$ and $Y$ are (uniquely) $n$-homogeneous/ultrahomogeneous, then so is $X \times_1 Y$.
\end{proof}

\begin{example}\label{circle-graph}
  For every $n\in\NN_+$, let $C_n:=\tuple{V_n,E_n,d_n}$ be the $n$-point circle graph with simple graph distance, where $V_n:=\{0,\dotsc,n-1\}$ is the vertex set, the set of edges $E_n:=\{\{i,i+1\} \mod n\colon 0\leqslant i< n\}$ only connects consecutive vertices as well as the first and last vertex with each other, and
  \[
    d_n\maps V_n\times V_n\to\left[0,\frac n2\right]\cap\NN\colon \tuple{i,j}\mapsto\min\{|i - j|,n - |i - j|\}
  \]
  counts the minimum number of edges in $E_n$ you have to cross to get from $i$ to $j$.

  The metric space $C_n$ is ultrahomogeneous.
\end{example}

\begin{proof}
  For any $C_n$, it is clear that the automorphism group $\Aut(C_n)$ contains all rotations around the vertex set $\Phi_i(k)=(i+k\mod n)$ and all reflections across a vertex $i\in V_n$, $\Psi_i(k)=(2i-k\mod n)$.
  
  Each isometry $\phi\colon A\to B$ with $A,B\subseteq C_n$ can be extended to at least one such rotation or reflection since, after choosing any two points $x\neq y\in A$ (that are not antipodal in the case of even $n$), all other points $z\in C_n$ can be uniquely determined from their distances to $x$ and $y$, and thus the same holds for $\phi(z)$.
  Therefore, it suffices to consider $\card{A}\leq 2$, and in those cases it is easy to see that rotating one point $x$ onto its image and then potentially reflecting across $\phi(x)$ will give an automorphism mapping $A$ to $B$.
\end{proof}

\section{Isosceles-free spaces} \label{sec:iso-free}

In the following, we will study metric spaces $X$ with the property that all distances from a given point are distinct, i.e. $d(x,y)\neq d(x,z)$ for all distinct $x,y,z\in X$. We will refer to such spaces as \emph{isosceles-free} spaces since the condition is equivalent to ``X does not contain any isosceles triangles''.
The isosceles-free spaces were introduced under the name \emph{star-rigid} by Janoš and Martin~\cite{JanosMartin78}.

\begin{observation}
	Every isosceles-free space is zero-dimensional, as observed by Hattori~\cite[Theorem~2]{Hattori90}:
	Every ball $B(x, r)$ has at most one point at the boundary, and every subspace $C(x, y) := \set{z: d(z, x) < d(z, y)}$ is clopen.
	Hence, if $B(x, r)$ has exactly one point $y$ at the boundary, $B(x, r) \cap C(x, y)$ is a basic clopen set, and otherwise $B(x, r)$ is already a basic clopen set.
\end{observation}

\begin{proposition}\label{unique-isometric-embedding}
	If $X$ is any metric space and $Y$ is isosceles-free, then for every $x\in X$, $y\in Y$ there exists at most one isometric embedding $f\colon X\rightarrow Y$ which maps $x$ to $y$.
\end{proposition}

\begin{proof}
	Pick two isometric embeddings $f,g$ such that $f(x)=y=g(x)$.
	Hence for any $x'\in X$ we have that $d(y, f(x'))=d(x, x')=d(y, g(x'))$. But this means that $f(x')=g(x')$ since otherwise, we would have a non-trivial isosceles triangle in $Y$.
\end{proof}

\begin{corollary} \label{unique-automorphism}
	Every $1$-homogeneous isosceles-free space $X$ is uniquely $1$-homogeneous, i.e. for every $x,y\in X$ there exists precisely one $f\in\Aut(X)$ such that $f(x) = y$.
	
	\begin{proof}
		Since $X$ is $1$-homogeneous, there exists at least one $f\in\Aut(X)$ mapping $x$ to $y$, and according to Proposition~\ref{unique-isometric-embedding}, there exists at most one.
	\end{proof}
\end{corollary}

\begin{proposition} \label{ultra-free}
	Every $1$-homogeneous isosceles-free space $X$ is ultrahomogeneous.
\end{proposition}

\begin{proof}
	Let $i\colon A\rightarrow B$ be an isometry from a finite subset $A\subset X$ to $B\subset X$. Choose any $x\in A$ and find the automorphism $f$ which maps $x$ to $i(x)$. Then, $i$ and $f|_A$ are both isometric embeddings from $A$ into $X$ which map $x$ to $i(x)$. According to Proposition~\ref{unique-isometric-embedding}, this means they are equal.
\end{proof}

\begin{remark}
	Observe that the proof of Proposition \ref{ultra-free} shows that we have homogeneity not only for finite substructures, but for all substructures of $X$. This is called \emph{absolute homogeneity} by Piotr Niemiec, studied in his recent preprint \cite{Niemiec}.
\end{remark}

Since $1$-homogeneity and ultrahomogeneity are equivalent for isosceles-free spaces, we will call them just \emph{homogeneous isosceles-free spaces}.

\begin{proposition} \label{unique-two-homog}
	A metric space $X$ is homogeneous isosceles-free if and only if it is uniquely $2$-homogeneous.
	
	\begin{proof}
		Recall that being uniquely $2$-homogeneous is equivalent to being $2$-homogeneous and uniquely $1$-homogeneous.
		Suppose $X$ is $1$-homogeneous and isosceles-free.
		By Proposition~\ref{ultra-free}, $X$ is even ultrahomogeneous.
		By Corollary~\ref{unique-automorphism}, $X$ is uniquely $1$-homogeneous.
		
		On the other hand, suppose that $X$ is uniquely $2$-homogeneous, and let $x, y, z \in X$ be such that $d(x, y) = d(x, z)$.
		By $2$-homogeneity there is $f \in \Aut(X)$ with $f(x) = x$ and $f(y) = z$.
		By unique $1$-homogeneity, $f$ is the unique automorphism fixing $x$, and so $f = \id_X$ and $y = z$, so $X$ is isosceles-free.
	\end{proof}
\end{proposition}

Recall that a \emph{Boolean group} is a group $G$ such that $g^2 = 1$, i.e. $g\inv = g$, for every $g \in G$.
It follows that $G$ is Abelian as $g h g\inv h\inv = (gh)(gh) = 1$ and so $gh = hg$ for every $g, h \in G$.

\begin{proposition} \label{iso-free_boolean}
	For every isosceles-free space $X$ the isometry group $\Aut(X)$ is Boolean.
\end{proposition}
\begin{proof}
	For every $f \in \Aut(X)$ and $x \in X$ we have $d(x, f(x)) = d(f(x), f(f(x)))$, and hence $x = f(f(x))$ since $X$ is isosceles-free.
	Hence $f^2 = \id_X$.
\end{proof}

\begin{observation} \label{thm:evaluation}
	Let $X$ be a metric space.
	For every $a \in X$ let 
	\begin{itemize}
		\item $D_a\maps X \to \Dist(X)$ denote the \emph{distance map} $x \mapsto d(x, a)$,
		\item $E_a\maps \Aut(X) \to X$ denote the \emph{evaluation map} $f \mapsto f(a)$.
	\end{itemize}
	We have the following reformulation of the properties considered.
	\begin{enumerate}
		\item $X$ is isosceles-free if and only if the maps $D_a$ are injective, and in that case it follows that the maps $E_a$ are injective.
		\item $X$ is $1$-homogeneous if and only if the maps $E_a$ are surjective, and in that case it follows that the maps $D_a$ are surjective.
		\item $X$ is uniquely $1$-homogeneous if and only if the maps $E_a$ are bijective.
		\item $X$ is homogeneous isosceles-free if and only if the maps $D_a$ and $E_a$ are bijective.
	\end{enumerate}
	
	\begin{proof}
		Clearly the maps $D_a$ being injective is essentially the definition of being isosceles-free.
		The maps $E_a$ are injective and surjective if and only if for every $x, y \in X$ there exists at most and at least, respectively, one $f \in \Aut(X)$ such that $f(x) = y$.
		Proposition~\ref{unique-isometric-embedding} says that the first option is true for isosceles-free spaces.
		Also for every $r \in \Dist(X)$ there are $x, y \in X$ with $d(x, y) = r$, and so if $X$ is $1$-homogeneous, for every $a \in X$ there is $f \in \Aut(X)$ with $f(x) = a$, and so $d(a, f(y)) = r$ and the maps $E_a$ are surjective.
		The rest is clear.
	\end{proof}
\end{observation}

\begin{corollary} \label{thm:power_two}
  For every finite homogeneous isosceles-free metric space $X$ we have $\card{X} = 2^m$ for some $m \in \omega$.
  
  \begin{proof}
    We have a bijection $E_a\maps \Aut(X) \to X$ and $\Aut(X)$ is a Boolean group by Proposition~\ref{iso-free_boolean}.
  \end{proof}
\end{corollary}

\begin{example}
	Let $X = \set{1, 2, 3, 4}$ and let $R = \set{a, b, c}$ where $a, b, c$ are any positive real numbers forming a triangle.
	There are exactly three decompositions of $X$ into two pairs of points: $Y_a = \set{\set{1, 2}, \set{3, 4}}$, $Y_b = \set{\set{1, 3}, \set{2, 4}}$, $Y_c = \set{\set{1, 4}, \set{2, 3}}$.
	We put $d(x, y) = r$ if and only if $\set{x, y} \in Y_r$, for $x \neq y \in X$ and $r \in R$.
	This gives a simple example of a homogeneous isosceles-free space, as in Figure~\ref{fig:tetraedr}.
\end{example}

	\begin{figure}[!ht]
		\centering
		
		\begin{tikzpicture}[every node/.style={draw,shape=circle,fill,minimum size=2pt,inner sep=0pt,outer sep=1pt}]
			\definecolor{c-d1}{RGB}{255,0,0}
			\definecolor{c-d2}{RGB}{0,255,0}
			\definecolor{c-d3}{RGB}{0,0,255}
			
			\node (v1) at (0, 0) {};
			\node (v2) at (1.3, -0.7) {};
			\node (v3) at (0, 1.5) {};
			\node (v4) at (-1.3, -0.7) {};
			
			\draw[c-d1, thick] (v1)--(v2);
			\draw[c-d1, thick] (v3)--(v4);
			\draw[c-d2, thick] (v1)--(v3);
			\draw[c-d2, thick] (v2)--(v4);
			\draw[c-d3, thick] (v1)--(v4);
			\draw[c-d3, thick] (v2)--(v3);

			\draw[color=c-d1,thick] (3,1)--(3.2,1);
			\node[draw=none, fill=none] at (3.5,1) {$a$};
			\draw[color=c-d2,thick] (3,0.5)--(3.2,0.5);
			\node[draw=none, fill=none] at (3.5,0.5) {$b$};
			\draw[color=c-d3,thick] (3,0)--(3.2,0);
			\node[draw=none, fill=none] at (3.5,0) {$c$};
		\end{tikzpicture}
		
		\caption{A four-point homogeneous isosceles-free space.}
		\label{fig:tetraedr}
	\end{figure}
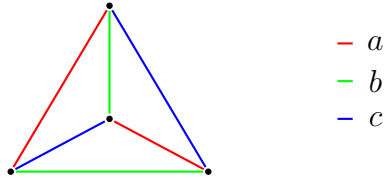

\begin{example}
  We consider the set $X=\{1,2,3,4,5,6\}$ and pick five pairwise distinct numbers $d_1,d_2,d_3,d_4,d_5\in [1,2]$ and set the distances as indicated in Figure~\ref{fig:hexagon}.
	
  Since all distances are between 1 and 2, the triangle inequality is always satisfied. Moreover at every point each distance appears exactly once, so the mappings $D_a$ are bijective and $X$ is isosceles-free but it is not $1$-homogeneous.
	That it is not $1$-homogeneous can be checked directly, but it also follows from Corollary~\ref{thm:power_two} as its size is not a power of two.
\end{example}

\begin{figure}[!ht]
		\centering
    \begin{tabular}{c|c|c|c|c|c|c}
      &   1   &   2   &   3   &   4   &   5   &   6   \\\hline
      1 &   0   & $d_1$ & $d_2$ & $d_3$ & $d_4$ & $d_5$ \\
      2 & $d_1$ &   0   & $d_5$ & $d_4$ & $d_2$ & $d_3$ \\
      3 & $d_2$ & $d_5$ &   0   & $d_1$  & $d_3$ & $d_4$ \\
      4 & $d_3$ & $d_4$ & $d_1$ &   0    & $d_5$ & $d_2$ \\
      5 & $d_4$ & $d_2$ & $d_3$ & $d_5$ &   0   & $d_1$ \\
      6 & $d_5$ & $d_3$ & $d_4$ & $d_2$ & $d_1$ &   0    \\      
    \end{tabular}
    \hspace{2cm}
    \raisebox{-1.2cm}{
      \begin{tikzpicture}[every node/.style={draw,shape=circle,fill,minimum size=2pt,inner sep=0pt,outer sep=1pt}]
        \definecolor{c-d1}{RGB}{255,0,0}
        \definecolor{c-d2}{RGB}{0,255,0}
        \definecolor{c-d3}{RGB}{255,165,0}
        \definecolor{c-d4}{RGB}{0,0,255}
        \definecolor{c-d5}{RGB}{200,200,0}                
        \node (v1) at (0,1.1) {};
        \node (v2) at (-1,0.5) {};
        \node (v3) at (-1,-0.5) {};
        \node (v4) at (0,-1.1) {};
        \node (v5) at (1,-0.5) {};
        \node (v6) at (1,0.5) {};
        \draw[c-d1, thick] (v1)--(v2);
        \draw[c-d1, thick] (v3)--(v4);
        \draw[c-d1, thick] (v5)--(v6);
        \draw[c-d2, thick] (v2)--(v3);
        \draw[c-d2, thick] (v4)--(v5);
        \draw[c-d2, thick] (v1)--(v6);
        \draw[c-d3, thick] (v1)--(v3);
        \draw[c-d3, thick] (v2)--(v5);
        \draw[c-d3, thick] (v4)--(v6);
        \draw[c-d4, thick] (v1)--(v4);
        \draw[c-d4, thick] (v2)--(v6);
        \draw[c-d4, thick] (v3)--(v5);
        \draw[c-d5, thick] (v1)--(v5);
        \draw[c-d5, thick] (v2)--(v4);
        \draw[c-d5, thick] (v3)--(v6);

        \draw[color=c-d1,thick] (2,1)--(2.2,1);
        \node[draw=none, fill=none] at (2.5,1) {$d_1$};
        \draw[color=c-d2,thick] (2,0.5)--(2.2,0.5);
        \node[draw=none, fill=none] at (2.5,0.5) {$d_2$};
        \draw[color=c-d3,thick] (2,0)--(2.2,0);
        \node[draw=none, fill=none] at (2.5,0) {$d_3$};
        \draw[color=c-d4,thick] (2,-0.5)--(2.2,-0.5);
        \node[draw=none, fill=none] at (2.5,-0.5) {$d_4$};
        \draw[color=c-d5,thick] (2,-1)--(2.2,-1);
        \node[draw=none, fill=none] at (2.5,-1) {$d_5$};
        
      \end{tikzpicture}}
		\caption{An isosceles-free space that is not $1$-homogeneous.}
		\label{fig:hexagon}
	\end{figure}
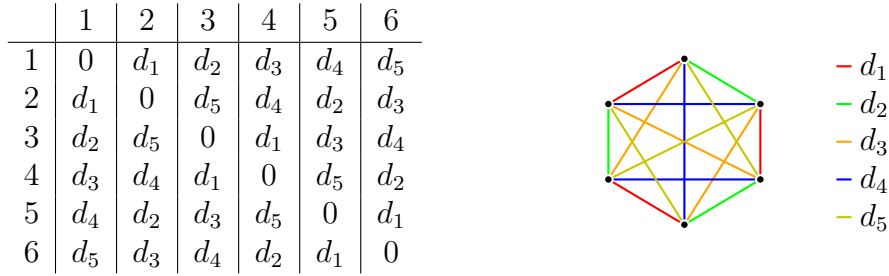

\begin{example}
	Let $X$ be a Polish (i.e. separable, complete) metric space in which all spheres and all bisectors are nowhere dense. A \emph{sphere} in $X$ is any set of the form
	\[S_r(a) := \{x \in X \colon d(a,x) = r\}\]
	while the \emph{bisector} of $a,b \in X$ is
	\[\bis(a,b) := \{x \in X \colon d(a,x) = d(x,b)\}.\]
	Assuming all spheres and all bisectors are nowhere dense, we can easily construct a sequence $A = \{a_n\}_{n \in \omega}$ such that all the distances between pairs of points of $A$ are pairwise distinct (such spaces are called \emph{strongly rigid}~\cite{Janos}).
	This way we obtain a dense countable isoceles-free subspace of $X$, where $X$ could be, for example, $\RR^n$, a manifold with the geodesic distance, a Banach space, or the Urysohn space.
\end{example}

The next result exhibits a universality property of the automorphism groups of homogeneous iso\-sce\-les-free spaces. Let us note that a countable ultrahomogeneous structure $U$, the Fraïssé limit of a given class of finite/finitely generated structures $\mathcal F$, is universal in the sense that it contains isomorphic copies of all countable structures that are unions of chains from $\mathcal F$. So, it is natural to ask whether $\Aut(U)$ contains isomorphic copies of $\Aut(X)$ for every $X \in \mathcal F$ or, even better, for every $X$ that is the union of a countable chain in $\mathcal F$. This universality question had been explicitly asked by Jaligot~\cite{Jaligot} and it turns out that for most classical Fraïssé classes the answer is positive~\cite{KubisMasulovic}, however there exist relational homogeneous structures whose automorphism groups are far from being universal, see~\cite{KubisShelah}. The next result gives a positive answer to the question above in the case of iso\-sce\-les-free metric spaces.

\begin{definition}
	Let $e\maps X \to Y$ be an isometric embedding of metric spaces.
	By an \emph{extension operator} along $e$ we mean a group homomorphism $e_*\maps \Aut(X) \to \Aut(Y)$ (which is necessarily injective) such that $e_*(f) \circ e = e \circ f$ for every $f \in \Aut(X)$. In the case that $e$ is the inclusion $X \subseteq Y$, this means simply that $e_*(f)$ extends $f$ for every $f \in \Aut(X)$.
\end{definition}

\begin{proposition} \label{extension_operator}
	Let $e\maps X \to Y$ be an isometric embedding of an isosceles-free space into a homogeneous isosceles-free space.
	\begin{enumerate}
		\item \label{itm:operator}
			There is a unique extension operator $e_*\maps \Aut(X) \to \Aut(Y)$.
		\item \label{itm:point_extension}
			For any $a \in X$ we have $e_* = E_{e(a)}\inv \circ e \circ E_a$.
		\item \label{itm:functor}
			The assignment $X \mapsto \Aut(X)$ and $e \mapsto e_*$ defines a functor from the category of homogeneous isosceles-free metric spaces and isometric embeddings to the category of Boolean groups and injective group homomorphisms.
	\end{enumerate}
	
	\begin{proof}
		If $e_*$ is an extension operator and $a \in X$, then for every $f \in \Aut(X)$ we have $e_*(f)(e(a)) = e(f(a))$, and by unique $1$-homogeneity of $Y$, $e_*(f)$ is the unique automorphism of $Y$ mapping $e(a)$ to $e(f(a))$, so $E_{e(a)}(e_*(f)) = e(E_a(f))$.
		This shows \ref{itm:point_extension} and uniqueness in \ref{itm:operator} if $X \neq \emptyset$.
		For $X = \emptyset$, we have $\Aut(X) = \set{\id_X}$ and \ref{itm:operator} holds.
		
		To show existence of the extension operator for $X \neq \emptyset$, we fix any $a \in X$ and for $f \in \Aut(X)$ we let $e_*(f)$ be the unique automorphism of $Y$ mapping $e(a) \mapsto e(f(a))$.
		Both $e_*(f) \circ e$ and $e \circ f$ map $a \mapsto e(f(a))$, and so they are equal by Proposition~\ref{unique-isometric-embedding}.
		For $f, g \in \Aut(X)$ we have $(e_*(f) \circ e_*(g)) \circ e = e_*(f) \circ e \circ g = e \circ (f \circ g)$, and hence $e_*$ is a group homomorphism $\Aut(X) \to \Aut(Y)$.
		Moreover the assignment is injective: if $e_*(f) = e_*(g)$, then $e \circ f = e \circ g$ and $f = g$ since $e$ is an embedding.
		
		To show \ref{itm:functor}, consider two isometric embeddings between homogeneous isosceles-free spaces $i\maps X \to Y$ and $j\maps Y \to Z$.
		For every $f \in \Aut(X)$ we have 
		\[
			(j_* \circ i_*)(f) \circ (j \circ i)
			= j_*(i_*(f)) \circ j \circ i
			= j \circ i_*(f) \circ i
			= j \circ (i \circ f),
		\]
		and clearly $j_* \circ i_*$ is a group homomorphism.
		Hence, $j_* \circ i_* = (j \circ i)_*$.
		Clearly also $(\id_X)_* = \id_{\Aut(X)}$.
	\end{proof}
\end{proposition}

In the following, we will talk about classes of metric spaces (with isometric embeddings as morphisms), and we define the weak amalgamation property, which was formally introduced by Ivanov~\cite{Ivanov} and independently by Kechris and Rosendal~\cite{KechrisRosendal} in connections with generic automorphisms of Fraïssé limits. It was recently explored in the context of Fraïssé limits by Krawczyk and Kubiś~\cite{KrawczykKubis}; a purely category-theoretic framework was developed in~\cite{KubisWFL} and for more information we refer to these two sources.
(WAP) is crucial for the existence of (weak) Fraïssé sequences and thus for the construction of an object $M$ which is \emph{generic} over $\K$, roughly speaking, the most common (or perhaps most complicated) object (a metric space, in our case) that can be built as the union of a chain in $\K$.

In the context of metric spaces, it seems to be rather difficult to find easy-to-describe classes that lack the weak amalgamation property. Note that graphs can be seen as metric spaces, with distance set $\{0,1,2\}$ depending on whether two points are connected by an edge or not. Then, graph embeddings correspond to isometric embeddings and thus some nontrivial examples of hereditary classes without (WAP) are given in \cite{KrawczykKubis} and \cite{MR4458208}.

\begin{definition}
	A class of metric spaces $\K$ has the \emph{weak amalgamation property} (WAP) if for every $A\in\K$ there exists a $\K$-embedding $\phi\maps A\to B$ such that for every two $\K$-embeddings $\psi_X\maps B\to X$, $\psi_Y\maps B\to Y$ there exist $\K$-embeddings $\pi_X\maps X\to Z$, $\pi_Y\maps Y\to Z$ into a common space $Z$ such that both ways of mapping $A$ to $Z$ coincide: $\pi_X\circ\psi_X\circ\phi=\pi_Y\circ\psi_Y\circ\phi$ (cf. Figure \ref{fig:WAP}).
	(Here a $\K$-embedding means an isometric embedding between spaces from $\K$.)
\end{definition}

\begin{figure}[ht!]
	\begin{center}
		\begin{tikzcd}
			&B\arrow[r,"\psi_X"]&X\arrow[rd,"\pi_X"]\\
			A\arrow[ru,"\phi"]\arrow[rd,"\phi"']&&&Z\\
			&B\arrow[r,"\psi_Y"]&Y\arrow[ru,"\pi_Y"']
		\end{tikzcd}\qquad
		\begin{tikzcd}
			&&X\arrow[rd,"\pi_X"]\\
			A\arrow[r,"\phi"]&B\arrow[ru,"\psi_X"]\arrow[rd,"\psi_Y"']&&Z\\
			&&Y\arrow[ru,"\pi_Y"']
		\end{tikzcd}
	\end{center}
	\caption{\label{fig:WAP}(WAP) requires the left diagram to be commutative, the right one need not be.}
\end{figure}
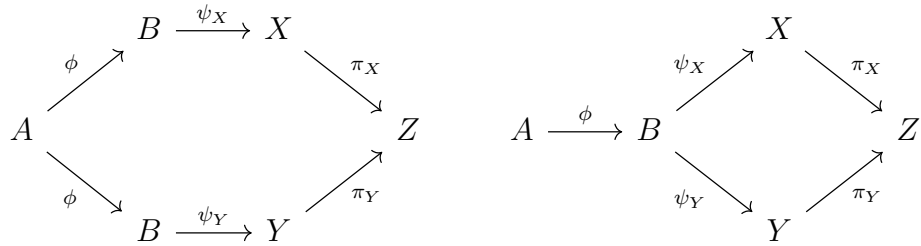

\begin{theorem} \label{WAP_fails}
	The class of all finite isosceles-free spaces does not have the weak amalgamation property.
\end{theorem}

\begin{proof}
	Let $\K$ denote the class of all finite isosceles-free spaces, let $A:=\{0,1\}$, let $\tuple{A,d_A}$ be our base space in $\K$, and let $\tuple{B,d_B}\in\K$ be any extension of $\tuple{A, d_A}$.
	In order to show a failure of (WAP), we will define two one-point extensions $X:=B\cup\{x\}$ and $Y:=B\cup\{y\}$ of $B$, with distance functions $d_X$ and $d_Y$, respectively, such that no space $Z$ exists in $\K$ which allows $X$ and $Y$ to be embedded into it in such a way that the images of $A$ coincide.
	
	To this end, define $r_0$ as a distance larger than any distance in $B$, and let $r_1$ be a distance very close to $r_0$:
	\begin{align*}
		r_0&:=\max\{d_B(a,b)\colon a,b\in B\}+1,\\
		\varepsilon&:=\frac12 \min\{|d_B(a,0)-d_B(b,1)|>0\colon a,b\in B\},\\
		r_1&:=r_0-\varepsilon.
	\end{align*}
	Note that $B$ is a finite metric space and thus the minimum in the definition of $\varepsilon$ really is a minimum rather than an infimum (in particular, $\varepsilon>0$).
	
	Since they are supposed to be extensions of $B$, let $d_X(a,b):=d_B(a,b)=:d_Y(a,b)$ for any $a,b\in B$. Moreover, for any $a\in B$, let
	\[d_X(a,x):=d_B(a,0)+r_0,\qquad d_Y(a,y):=\min\{d_B(a,0)+r_0,d_B(a,1)+r_1\}\]
	Clearly, this way, $X$ is a valid finite metric space and both choices of distances $\rho(a,y):=d_B(a,0)+r_0$ and $\rho'(a,y):=d_B(a,1)+r_1$ within the minimum in the definition of $d_Y$ would define a valid metric on $Y$ (again with $\rho=\rho'=d_B$ on $B\times B$). So to show that $d_Y$ is valid as well, observe that the minimum of two metrics always satisfies all conditions for a metric except potentially the triangle inequality. However, the two metrics $\rho$ and $\rho'$ only differ when $y$ is one of the two arguments, so we need to check whether
	\begin{align*}
		d_Y(a,y)=\min\{\rho(a,y),\rho'(a,y)\}&\leqslant\min\{\rho(a,b)+\rho(b,y),\rho'(a,b)+\rho'(b,y)\}\\
		&= d_B(a,b)+\min\{\rho(b,y),\rho'(b,y)\}.
	\end{align*}
	But since $\rho$ and $\rho'$ are valid metrics on $Y$, the only cases in which this might fail are when (w.l.o.g.) $\rho(a,y)<\rho'(a,y)$ and $\rho(b,y)>\rho'(b,y)$. But in that case,
	\begin{align*}
		d_Y(a,y)<\rho'(a,y)\leqslant\rho'(a,b)+\rho'(b,y)=d_Y(a,b)+d_Y(b,y).
	\end{align*}
	
	Lastly, when $y$ only shows up on the right-hand side of the triangle inequality, we have to show that:
	\begin{align*}
		d_Y(a,b)\leqslant\min\{\rho(a,y),\rho'(a,y)\}+\min\{\rho(y,b),\rho'(y,b)\}=d_Y(a,y)+d_Y(y,b).
	\end{align*}
	If $\rho(a,y)<\rho'(a,y)$ or $\rho(b,y)<\rho'(b,y)$ this is clear since $\rho(\,\cdot\,,y)\geqslant r_0>d_B(a,b)=d_Y(a,b)$ on $B$. If, on the other hand, $\rho(a,y)>\rho'(a,y)$ and $\rho(b,y)>\rho'(b,y)$, then
	\[d_Y(a,y)+d_Y(b,y)=\rho'(a,y)+\rho'(b,y)\geqslant\rho'(a,b)=d_Y(a,b).\]
	
	So $(X,d_X),(Y,d_Y)$ are both valid metric spaces. Let us additionally show that they are isosceles-free: since $B$ already contains no isosceles triangles, the only way to add an isosceles triangle to $X$ would be if $d_X(a,x)=d_X(b,x)$ for some $a,b\in B$, but in that case, it would follow that $d_B(a,0)=d_B(b,0)$, a clear contradiction.
	
	For $Y$, the situation is slightly more complicated. If $a,b\in B$ and $d_Y(a,y)=d_Y(b,y)$, then we need to consider four cases, however, if both $d_Y(a,y)=d_B(a,0)+r_0$ and $d_Y(b,y)=d_B(b,0)+r_0$ or $d_B(a,1)+r_1=d_Y(b,y)=d_B(b,1)+r_1$, then the same argument as in the previous paragraph leads to the conclusion that $B$ is not isosceles-free. Thus, up to re-labelling $a$ and $b$, we only need to consider the case where
	\[d_Y(a,y)=d_B(a,0)+r_0=d_Y(b,y)=d_B(b,1)+r_1.\]
	
	However, in this case,
	\begin{align*}
		d_B(a,0)+r_0=d_B(b,1)+r_0-\varepsilon\Leftrightarrow \varepsilon=d_B(b,1)-d_B(a,0).
	\end{align*}
	Clearly, this cannot be the case due to our definition of $\varepsilon$ since the right-hand side is always either $0$, negative or at least double the value of $\varepsilon$.
	
	It follows that $X$ and $Y$ are both in $\K$. However, in order for (WAP) to hold, we would need to find a $Z\in\K$ such that both $X$ and $Y$ embed into $Z$ in such a way that $0$ and $1$ in $A$ are mapped to the same points in $Z$ no matter whether they are mapped via $X$ or via $Y$.
	
	So any such space $Z$ would need to satisfy that there exist $\K$-embeddings $\pi_X,\pi_Y$ such that $\pi_X(0)=\pi_Y(0)=:0$ and $\pi_X(1)=\pi_Y(1)=:1$.
	
	However, in such a case, due to isometry of $\K$-embeddings, we get
	\[d_Z(0,\pi_X(x))=d_X(0,x)=r_0=\min\{r_0,d_B(0,1)+r_1\}=d_Y(0,y)=d_Z(0,\pi_Y(y)).\]
	
	This contradicts our assumption that $Z\in\K$ since for that, $Z$ would have to be iso\-sce\-les-free and yet $\pi_X(x)\neq\pi_Y(y)$ (their distances to $1$ are different, for example) and $\{0,\pi_X(x),\pi_Y(y)\}$ is a non-trivial isosceles triangle in $Z$.
\end{proof}

\begin{remark}
  Note that the result above is valid when the distance set is restricted to a dense subgroup $A$ of $\RR$, which includes the case of rational distances.
  Let $\mathcal F_A$ denote the class of all countable isosceles-free spaces $X$ with $\Dist(X) \subseteq A$.
  
  One of the properties of a Fraïssé limit is that it is universal for the associated class of countable structures. When (WAP) fails, not only is there no Fraïssé limit, but there is not even a universal structure; in fact by~\cite[Corollary~6.3]{KrawczykKubis} the universality number of $\mathcal F_A$, that is the minimal cardinality of a subfamily $\mathcal C \subseteq \mathcal F_A$ such that every  $X \in \mathcal F_A$ embeds isometrically into a member of $\mathcal C$, is the continuum.
\end{remark}

\begin{remark} \label{rm:amalgamation}
	Note that by classical Fraïssé theory~\cite[Theorem~7.1.7]{Hodges}, for every countable homogeneous isosceles-free metric space $X$, the family of all finite spaces embeddable into $X$ (denoted by $\Age(X)$) has even the \emph{amalgamation property} (AP).
	This is no contradiction with the previous result – $\Age(X)$ is a much more restrictive class of finite isosceles-free spaces than $\mathcal F_A$ from the previous remark.
	In fact, in a homogeneous isosceles-free space $X$ for every positive $p \neq q \in \Dist(X)$ there is $r \in \Dist(X)$ such that every triangle in $X$ with distances $p$ and $q$ is completed by the distance $r$.
	Hence, every one-point extension $F \cup \set{x} \subseteq X$ is uniquely determined by $d_F$ and a single distance $d(a, x)$ for a fixed point $a \in F$ since every $d(b, x)$ for $b \neq a \in F$ is the unique distance completing the distances $d(a, x)$ and $d(a, b)$.
	This also shows that the class of finite homogeneous isosceles-free spaces does not have the \emph{joint embedding property} (JEP), while it is easy to see that the class of finite isosceles-free spaces has (JEP).
	We will give a precise description of $\Age(X)$ for a homogeneous isosceles-free space $X$ in Proposition~\ref{thm:isosceles-free_amalgamation_class} and Corollary~\ref{thm:isosceles-free_amalgamation_class_corollary}.
\end{remark}

\section{Boolean metric spaces} \label{sec:boolean}

\begin{definition}
  By a \emph{Boolean metric space} we mean a nonempty $1$-homogeneous metric space $X$ such that $\Aut(X)$ is a Boolean group.
\end{definition}

\begin{remark}
	Note that the notion of Boolean metric space we use here is not related to the notion where the metric itself takes values in a Boolean algebra, as used by, for example, Melter in \cite{MelterRobertA.1964Bvra} or Avilés in \cite{AvilesAntonio2004BMSa}.
\end{remark}

We have shown that every homogeneous isosceles-free space is Boolean (Proposition~\ref{iso-free_boolean}) and uniquely $1$-homogeneous (Corollary~\ref{unique-automorphism}).
It turns out that every Boolean metric space is uniquely $1$-homogeneous and that it is in fact enough to suppose that the automorphism group is Abelian (Corollary~\ref{cor:Boolean}).
Moreover, Boolean metric spaces can be viewed as normed $\ZZ_2$-linear spaces, which gives us a concrete representation of every homogeneous isosceles-free space.

\medskip

By a \emph{norm} on an Abelian group $X$ we mean a map $\norm{\cdot}\maps X \to [0, \infty)$ such that
\begin{enumerate}
	\item $\norm{x} = 0$ if and only if $x = 0$, for $x \in X$,
	\item $\norm{x + y} \leq \norm{x} + \norm{y}$, for $x, y \in X$,
	\item $\norm{-x} = \norm{x}$, for $x \in X$.
\end{enumerate}
It is a more general version of an F-norm in linear vector spaces, cf. Rolewicz~\cite[p.~4]{Rolewicz}, and is nowadays present in several aspects of group theory.
It is well-known that every norm induces an \emph{invariant metric} on $X$ (i.e. a metric such that $d(x + z, y + z) = d(x, y)$ for every $x, y, z \in X$) by putting $d(x, y) := \norm{x - y}$.
Then we have $\norm{x} = d(x, 0)$.
On the other hand, the previous formula gives a norm for any invariant metric on $X$.
Altogether, norms and invariant metrics on $X$ are in one-to-one correspondence.
Similarly, norm-preserving maps $X \to Y$ between normed Abelian groups are in one-to-one correspondence with isometric embeddings preserving $0$.

If the Abelian group $X$ is Boolean, $X$ is a linear space over $\ZZ_2$ and the norm is trivially a $\ZZ_2$-norm, i.e. it also satisfies $\norm{\alpha \cdot x} = \abs{\alpha} \cdot \norm{x}$ for every $\alpha \in \ZZ_2$ and $x \in X$.
Altogether, a Boolean group endowed with an invariant metric is the same thing as a normed $\ZZ_2$-linear space, and isometric embeddings preserving $0$ are linear.

\begin{proposition} \label{thm:Boolean}
	Let $X$ be a $1$-homogeneous space such that $\Aut(X)$ is Abelian.
	\begin{enumerate}
		\item For every $f \in \Aut(X)$, the displacement $d(x, f(x))$ does not depend on the point $x \in X$.
		\item \label{itm:norm}
			Putting $\norm{f} := d(x, f(x))$ for any $x \in X$ defines a norm on $\tuple{\Aut(X), \circ}$.
		\item \label{itm:unique_1_homog}
			$X$ is uniquely $1$-homogeneous.
		\item \label{itm:Aut_isometric}
			The evaluation map $E_a\maps \Aut(X) \to X$ is an isometry for every $a \in X$.
		\item $\Aut(X)$ is a Boolean group.
	\end{enumerate}
	
	\begin{proof} \hfill
		\begin{enumerate}
			\item
				Let $f \in \Aut(X)$ and $x, y \in X$.
				By $1$-homogeneity there is a $g \in \Aut(X)$ with $g(x) = y$.
				We have $d(y, f(y)) = d(g(x), f(g(x))) = d(g(x), g(f(x))) = d(x, f(x))$.
			\item
				We have $\norm{f} = 0$ if and only if $d(x, f(x)) = 0$ for every $x \in X$, i.e. if and only if $f = \id_X$.
				For every $f, g \in \Aut(X)$ and $x \in X$ we have $\norm{f \circ g} = d(x, f(g(x))) \leq d(x, f(x)) + d(f(x), f(g(x))) = d(x, f(x)) + d(x, g(x)) = \norm{f} + \norm{g}$.
				Finally, $\norm{f\inv} = d(x, f\inv(x)) = d(f(x), x) = \norm{f}$.
			\item
				If $f(x) = g(x)$ for $f, g \in \Aut(X)$ and some $x \in X$, then $\norm{f\inv \circ g} = d(x, f\inv(g(x))) = d(f(x), g(x)) = 0$, and so $f\inv \circ g = \id_X$ by the first property of the norm.
			\item
				We have $d(E_a(f), E_a(g)) = d(f(a), g(a)) = d(a, f\inv(g(a))) = \norm{f\inv \circ g} = d_{\Aut(X)}(f, g)$.
				Hence, $E_a$ is an isometric embedding.
				It is onto since $X$ is $1$-homogeneous (see Observation~\ref{thm:evaluation})
			\item
				By \ref{itm:Aut_isometric}, $\Aut(X)$ is isometric to $X$, and so is uniquely $1$-homogeneous by \ref{itm:unique_1_homog}.
				By the third property of the norm, the map $\phi\maps \Aut(X) \to \Aut(X)$, $f \mapsto f\inv$, is norm-preserving, and hence an isometry fixing $\id_X$.
				Therefore $\phi(\id_X) = \id_{\Aut(X)}(\id_X)$, and so $\phi = \id_{\Aut(X)}$ since $\Aut(X)$ is uniquely $1$-homogeneous.
			\qedhere
		\end{enumerate}
	\end{proof}
\end{proposition}

\begin{corollary} \label{cor:Boolean}
	For a $1$-homogeneous metric space $X$, $\Aut(X)$ is Abelian if and only if $\Aut(X)$ is Boolean, and in this case, $X$ is uniquely $1$-homogeneous.
\end{corollary}

\begin{corollary}
	Let $X$ be a Boolean metric space.
	$\Aut(X)$ is a normed $\ZZ_2$-linear space, and the canonical action of $\Aut(X)$ on $X$ turns $X$ into an affine space over $\Aut(X)$.
	Moreover, every evaluation map $E_a\maps \Aut(X) \to X$ is an affine isometry.
	
	\begin{proof}
		$\Aut(X)$ is a Boolean group, and hence a $\ZZ_2$-linear space.
		By Proposition~\ref{thm:Boolean}~\ref{itm:norm} it is endowed with a norm, which is trivially a $\ZZ_2$-norm.
		By Proposition~\ref{thm:Boolean}~\ref{itm:unique_1_homog} the action of $\Aut(X)$ on $X$ is transitive and faithful, and so $X$ is an affine space over $\Aut(X)$.
		By Proposition~\ref{thm:Boolean}~\ref{itm:Aut_isometric} the map $E_a$ is an isometry.
		Moreover, it is affine since its linear part is just $\id_{\Aut(X)}$.
	\end{proof}
\end{corollary}

Since every Boolean group is a $\ZZ_2$-linear space, and by choosing a basis $I$ we obtain an isomorphism to $\ZZ_2^{(I)}$, i.e. to the subspace of $\ZZ_2^I$ consisting of all functions of finite support.
We can equivalently view $\ZZ_2^{(I)}$ as the family $\powfinset{I}$ of all finite subsets of $I$ with the operation of symmetric difference: $A \symdiff B = (A \setminus B) \cup (B \setminus A)$.
We shall write just $2^{(I)}$ and switch the perspective between $\tuple{\ZZ_2^{(I)}, +}$ and $\tuple{\powfinset{I}, \symdiff}$ as convenient, and similarly for $2^I$.
To turn $2^{(I)}$ into a normed $\ZZ_2$-linear space means to provide a map $\norm{\cdot}\maps 2^{(I)} \to [0, \infty)$ satisfying $\norm{\emptyset} = 0$ and $\norm{x \symdiff y} \leq \norm{x} + \norm{y}$ for $x, y \in 2^{(I)}$.

\begin{definition}
	We say that a metric space $X$ is \emph{$\ZZ_2$-normable} if it is isometric to a normed $\ZZ_2$-linear space, or equivalently to $\tuple{2^{(I)}, \norm{\cdot}}$ for some set $I$ and a norm $\norm{\cdot}$.
	Note that every $\ZZ_2$-normable space is $1$-homogeneous as witnessed by the translations.
\end{definition}

\begin{observation} \label{thm:properties}
	We have shown that every (nonempty) homogeneous isosceles-free space is Boolean, that every Boolean metric space is $\ZZ_2$-normable, and that $\ZZ_2$-normable spaces admit a very concrete description.
	Figure~\ref{fig:properties} summarizes the implications between the properties considered.
	
	Moreover, for a metric space $X$ we have the following.
	\begin{enumerate}
		\item $X$ is Boolean if and only if it is $\ZZ_2$-normable and uniquely $1$-homogeneous.
		\item $X$ is homogeneous isosceles-free if and only if it is Boolean and $2$-homogeneous.
	\end{enumerate}
	
	Also, a discrete metric space ($d(x, y) = 1$ for $x \neq y$) of size $2^n$ for $n \geq 2$ is clearly ultrahomogeneous and $\ZZ_2$-normable, but not uniquely $1$-homogeneous.
	Example~\ref{Dn} gives a uniquely $1$-homogeneous space that is not Boolean (or equivalently not $\ZZ_2$-normable).
	Example~\ref{ex:boolean_rainbow} gives a Boolean metric space that is not isosceles-free.
\end{observation}

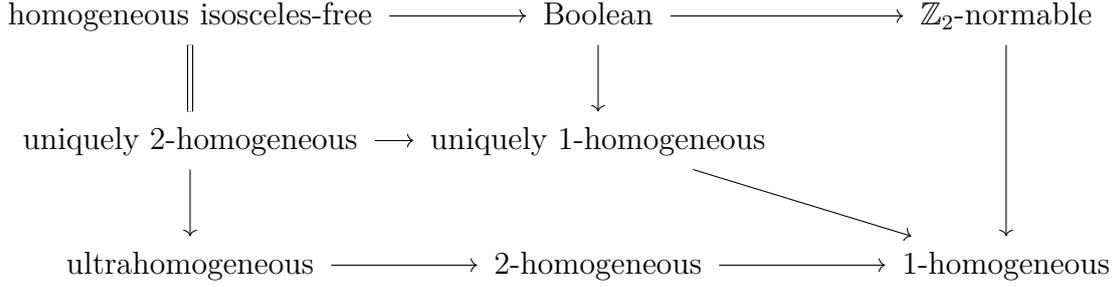
\begin{figure}[ht!]
	\centering
	\begin{tikzpicture}[
			x = {(13em, 0)},
			y = {(0, -8em)},
			text height = 1.5ex,
			text depth = 0.25ex,
			inner sep = 1.2ex,
		]
		\node (iso-free) at (0, 0) {homogeneous isosceles-free};
		\node (bool) at (1, 0) {Boolean};
		\node (norm) at (2, 0) {$\ZZ_2$-normable};
		\node (ultrahom) at (0, 1) {ultrahomogeneous};
		\node (2-hom) at (1, 1) {$2$-homogeneous};
		\node (1-hom) at (2, 1) {$1$-homogeneous};
		\node (u1-hom) at (1, 0.5) {uniquely $1$-homogeneous};
		\node (u2-hom) at (0, 0.5) {uniquely $2$-homogeneous};
		
		\graph{
			(iso-free) -> (bool) -> (norm),
			(ultrahom) -> (2-hom) -> (1-hom),
			(u2-hom) -> {(ultrahom), (u1-hom)},
			(norm) -> (1-hom),
			(bool) -> (u1-hom) -> (1-hom),
			(iso-free) --[double, double distance=1.5pt] (u2-hom),
		};
	\end{tikzpicture}
	
	\caption{Implications between properties of metric spaces considered.}
	\label{fig:properties}
\end{figure}

\begin{proof}
	Suppose that $X$ is $\ZZ_2$-normable.
	Then for every $x, y \in X$ there is a unique translation $T \in \Aut(X)$ such that $T(x) = y$.
	If $X$ is also uniquely $1$-homogeneous, then all auto-isometries are translations, and so $\Aut(X)$ is Abelian, and we may use Corollary~\ref{cor:Boolean}.
	
	Now suppose that $X$ is Boolean and $2$-homogeneous.
	Then $X$ is uniquely $2$-ho\-mo\-ge\-ne\-ous, and so isosceles-free by Proposition~\ref{unique-two-homog}.
\end{proof}

We observe that it is easy to identify isosceles-free spaces among $\ZZ_2$-normable spaces.
\begin{observation}
	A normed linear space $X$ (a priori over any valued field) is isosceles-free if and only if the norm $\norm{\cdot}\maps X \to [0, \infty)$ is injective, and in that case the field is necessarily $\ZZ_2$ since $\set{x, 0, -x}$ forms an isosceles triangle unless $x = -x$.
\end{observation}

Altogether, we obtain the following summarizing theorem.
\begin{theorem} \label{thm:homog_iso-free}
	Let $I$ be a set and let $\norm{\cdot}\maps 2^{(I)} \to [0, \infty)$ be an injective map satisfying $\norm{\emptyset} = 0$ and $\norm{x \symdiff y} \leq \norm{x} + \norm{y}$ for $x, y \in 2^{(I)}$.
	By putting $d(x, y) := \norm{x \symdiff y}$ for $x, y \in 2^{(I)}$ we obtain a homogeneous isosceles-free space.
	Moreover, every homogeneous isosceles-free space can be obtained this way up to an isometry.
\end{theorem}

\begin{definition}
	We say that a $\ZZ_2$-normable space $X$ is \emph{additive} if it is isometric to the space $2^{(I)}$ with the norm $\norm{x} = \sum_{i \in x} r_i = \sum_{i \in I} x(i) \cdot r_i$ for some $\set{r_i: i \in I} \subseteq (0, \infty)$.
	We always have $\norm{x \symdiff y} \leq \norm{x} + \norm{y}$ in this case.
	The distance satisfies $d(x, y) = \sum_{i \in I} \abs{x(i) - y(i)} \cdot r_i$, so $X$ embeds into $\ell_1(I)$.
	
	We say that a $\ZZ_2$-normable space $X$ is \emph{monotone} if it is isometric to the space $2^{(I)}$ with a monotone norm, i.e. $\norm{x} \leq \norm{y}$ for every $x \subseteq y$.
	Clearly, every additive $\ZZ_2$-normable space is monotone.
\end{definition}

\begin{remark}
	The triangle inequality of the space $\tuple{2^{(I)}, \norm{\cdot}}$ expressed using the norm is
	\[
		\norm{x \symdiff z} \leq \norm{x \symdiff y} + \norm{y \symdiff z} \quad \text{for every $x, y, z \in 2^{(I)}$},
	\]
	but it reduces to
	\[
		\norm{x' \symdiff y'} \leq \norm{x'} + \norm{y'} \quad \text{for every $x, y \in 2^{(I)}$}
	\]
	in every Boolean group with an invariant metric since we can put $x' = x \symdiff y$ and $y' = y \symdiff z$.
	
	A different simplification of the original inequality is the fact that it is enough to verify it only for triples of pairwise disjoint sets.
	For every $\set{x_k: k < 3} \subseteq X$ we put $w := \set{i \in I: \card{\set{k < 3: i \in x_k}} \geq 2}$.
	Then $w$ is a finite set, the sets $x_k \symdiff w$, $k < 3$, are pairwise disjoint, and $d(x_k, x_{k'}) = d(x_k \symdiff w, x_{k'} \symdiff w)$.
	The triangle inequality then reduces to 
	\[
		\norm{x \cup z} \leq \norm{x \cup y} + \norm{y \cup z} \quad \text{for every pairwise disjoint $x, y, z \in 2^{(I)}$.}
	\]
	In particular, the norm is $\cup$-subadditive: $\norm{x \cup y} \leq \norm{x} + \norm{y}$ for $x, y$ disjoint, but it may not be monotone (see Example~\ref{ex:nonmonotone}).
\end{remark}

\begin{example} \label{ex:power-of-two}
	Let $r_n := 2^n$ for $n \in \omega$.
	The induced additive norm is the bijection $2^{(\omega)} \to \omega$ corresponding to binary expansions of natural numbers.
	Hence we obtain the countable infinite discrete homogeneous isosceles-free space $X_\omega = \tuple{2^{(\omega)}, \norm{\cdot}}$.
	We can also consider the restricted finite spaces $X_n = \tuple{2^n, \norm{\cdot}}$.
	We have $\Dist(X_n) = \set{0, \ldots, 2^n - 1}$.
\end{example}

\begin{example} \label{ex:Cantor2}
	Let $r_n := 2^{-n}$ for $n \in \omega$ and let $\norm{\cdot}$ be the corresponding additive norm on $2^{(\omega)}$, which is a bijection onto the dyadic rational numbers in $[0, 2)$.
	The corresponding homogeneous isosceles-free space is the dense subset $2^{(\omega)}$ of the Cantor space $2^\omega$ with the metric $d(x, y) = \sum_{n \in \omega} \abs{x(n) - y(n)} \cdot 2^{-n}$.
	
	Note that the completion $2^\omega$ of our homogeneous isosceles-free space $2^{(\omega)}$ is uniquely $1$-homogeneous and Boolean, but not $2$-homogeneous and not isosceles-free.
	The $1$-ho\-mo\-gene\-ity follows from the fact that $2^\omega$ is a normed $\ZZ_2$-linear space.
	Hence, $2^\omega$ is Boolean if and only if it is uniquely $1$-homogeneous.
	If $f(x) = g(x)$ for $x \in 2^\omega$ and $f, g \in \Aut(2^\omega)$, then $f(h(0)) = g(h(0))$ for $h \in \Aut(2^\omega)$ with $h(0) = x$, and so $h\inv \circ f\inv \circ g \circ h$ is a an auto-isometry fixing $0$.
	It is enough to show that $\id_{2^\omega}$ is the only auto-isometry $\phi$ fixing $0$.
	This follows from the fact that every $x \in 2^\omega$ that is not eventually constant is the unique element of norm $\norm{x}$, and so $x = \id$ on a dense subset of $2^\omega$.
	
	Let $e_n, e'_n \in 2^\omega$ denote the characteristic function of $\set{n}$ and of its complement, respectively.
	Then $\norm{e_0} = 1 = \norm{e'_0}$, so the completion is not isosceles-free.
	Also, $e_1$ is a mid-point of $0$ and $e'_0$, while there is no mid-point of $0$ and $e_0$, and so the completion is not $2$-homogeneous.
\end{example}

\begin{example} \label{ex:Cantor3}
	Let $r_n := 3^{-n}$ for $n \in \omega$ and let $\norm{\cdot}$ be the corresponding additive norm on $2^{(\omega)}$, which is injective.
	Similarly to the previous example, the corresponding homogeneous isosceles-free space is the dense subset $2^{(\omega)}$ of the Cantor space $2^\omega$ with the metric $d(x, y) = \sum_{n \in \omega} \abs{x(n) - y(n)} \cdot 3^{-n}$.
	However, this time the completion $2^\omega$ is still a homogeneous isosceles-free space.
	This is because the norm $\norm{x} = \sum_{n \in \omega} x(n) \cdot 3^{-n}$ is injective on $2^\omega$:
	if $x(n) = y(n)$ for every $n < n_0$ and $x(n_0) < y(n_0)$, then $\norm{x} \leq \sum_{n > n_0} 3^{-n} = 3/2 \cdot 3^{-(n_0 + 1)} < 3^{-n_0} \leq \norm{y}$.
\end{example}

\begin{remark}
	It is known that the completion of a countable ultrahomogeneous metric sometimes is (as for the rational Urysohn space) and sometimes is not (see \cite[Proposition~10]{NguyenVanThe}) ultrahomogeneous.
	The two very similar examples above demonstrate this phenomenon in the realm of isosceles-free homogeneous spaces.
\end{remark}

In the next proposition we refine our results on extension operators (Proposition~\ref{extension_operator}).
\begin{proposition} \label{thm:linear_maps}
	Let $X$ and $Y$ be homogeneous isosceles-free spaces.
	\begin{enumerate}
		\item For every isometric embedding $e\maps X \to Y$ the extension operator $e_*\maps \Aut(X) \to \Aut(Y)$ is a linear isometric embedding.
		\item Every isometric embedding $\Aut(X) \to \Aut(Y)$ mapping $\id_X$ to $\id_Y$ is an extension operator and hence linear, and every isometric embedding $X \to Y$ is affine.
	\end{enumerate}
	
	\begin{proof}
		For every $a \in X$ and $f \in \Aut(X)$, $e_*(f)$ maps $e(a)$ to $e(f(a))$, and so we have $\norm{e_*(f)} = d(e(a), e(f(a))) = d(a, f(a)) = \norm{f}$.
		As a group homomorphism, $e_*$ is linear.
		Together, a norm-preserving linear map is an isometric embedding.
				
		For every embedding $e\maps X \to Y$ we have that $e = E_{e(a)} \cmp e_* \cmp E_a\inv$ by Proposition~\ref{extension_operator}, and hence is affine as a composition of affine maps.
		On the other hand, every isometric embedding $f\maps \Aut(X) \to \Aut(Y)$ such that $f(\id_X) = \id_Y$ is of the form $e_*$, and so is linear.
		Namely, we take any $a \in X$ and $b \in Y$ and put $e := E_b \cmp f \cmp E_a\inv$.
		Since $f(\id_X) = \id_Y$, we have $e(a) = E_b(f(\id_X)) = E_b(\id_Y) = b$, and so $f = e_*$ by unique $1$-homogeneity.
	\end{proof}
\end{proposition}

\begin{example} \label{ex:nonmonotone}
	There is a homogeneous isosceles-free space $X$ that is not monotone.
	We take $X$ to be $2^{\set{0, 1, 2}}$ and define the norm so that $\norm{0} = 0$, $\set{\norm{e_0}, \norm{e_1}, \norm{e_2}} = \set{10, 11, 12}$, $\set{\norm{e_0 + e_1}, \norm{e_1 + e_2}, \norm{e_2 + e_0}} = \set{14, 15, 16}$, and $\norm{e_1 + e_2 + e_3} = 13$.
	The norm is injective, and the triangle inequality is satisfied since the positive distances are in the interval $[a, 2a]$ for $a = 10$, so we have a homogeneous isosceles-free space.
	The given norm is obviously not monotone, but we need to show that the space is not isometric to $2^{\set{0, 1, 2}}$ with a monotone norm.
	By homogeneity, it is enough to consider isometries fixing the zero vector, and such isometries are linear by Proposition~\ref{thm:linear_maps}.
	$\set{e_0, e_1, e_2}$ is a linearly independent set of vectors of norm taking the smallest positive values from $\Dist(X)$, and hence every isomorphism making the norm monotone would need to fix this set up to re-labelling.
	However, the vector $e_0 + e_1 + e_2$ would need to be fixed as well by linearity, and so the other norm would not be monotone.
\end{example}

In the last part of this section we describe amalgamation classes associated to homogeneous isoceles-free spaces, as promised in Remark~\ref{rm:amalgamation}.

Let us call a triple $\tuple{p, q, r}$ of positive real numbers a \emph{triangle} if $p \leq q + r$ and $q \leq r + p$ and $r \leq p + q$.
If $p = q$, then $0$ is the only $r$ completing the triangle.
Similarly, if $p = 0$, then $r = q$ is the only choice completing the triangle.
These triangles are called \emph{degenerate}, while triangles with $p, q, r > 0$ are \emph{non-degenerate}.
Let $T_{p, q, r}(a, b, c)$ denote the space $\set{a, b, c}$ with the metric $d(a, b) = p$, $d(a, c) = q$, $d(b, c) = r$.
This notation automatically implies that if $p = 0$, then $a = b$, and so on.
We write just $T_{p, q, r}$ when the supporting set is irrelevant.

For a class $\F$ of metric spaces we denote the set of distances $\bigcup_{X \in \F} \Dist(X) \subseteq [0, \infty)$ by $\Dist(\F)$.
We also call $\F$ \emph{hereditary} if for every isometric embedding $e\maps A \to B$ with $B \in \F$ we have $A \in \F$.
This automatically means that $\F$ is closed under isomorphic copies.

\begin{proposition} \label{thm:isosceles-free_amalgamation_class}
  Let $\F$ be a class of isosceles-free metric spaces such that for every $p, q \in \Dist(\F)$ there is $A \in \F$ and $a, b, c \in A$ such that $d(a, b) = p$ and $d(a, c) = q$.
  Then the following conditions are equivalent.
  \begin{enumerate}
  \item \label{itm:class}
    $\F$ can be extended to a hereditary class $\overline{\F}$ of isosceles-free spaces with the amalgamation property such that $\Dist(\overline{\F}) = \Dist(\F)$.
  \item \label{itm:coherence}
    We have \begin{enumerate}
    \item \label{itm:coherence_a}
      for every $p, q \in \Dist(\F)$ there is a unique $t(p, q) \in \Dist(\F)$ such that $\tuple{p, q, t(p, q)}$ is a triangle and $T_{p, q, t(p, q)}$ embeds into a member of $\F$,
    \item \label{itm:coherence_b}
      for every $p, q, p', q' \in \Dist(\F)$ with $t(p, q) = t(p', q')$ we have $t(p, p') = t(q, q')$.
    \end{enumerate}
  \item \label{itm:limit}
    There is a homogeneous isosceles-free space $X_{\F}$ such that $\Dist(X_\F) = \Dist(\F)$ and $\Age(X_\F) \supseteq \F$.
  \end{enumerate}
  Moreover, the class $\overline{\F}$ and the space $X_\F$ are unique, and $\overline{\F} = \Age(X_\F)$.
  
  \begin{proof}
    Suppose \ref{itm:class} and let $p, q \in \Dist(\F)$. We show that \ref{itm:coherence} holds.
    By the assumption there is some $r$ such that $T_{p, q, r}$ is embedded into a member of $\F$.
    Since $\overline{\F}$ is hereditary, we have $T_{p, q, r}(a, b, c) \in \overline{\F}$.
    If $T_{p, q, r'} \in \overline{\F}$ for some $r' \neq r$, without loss of generality, we have $T_{p, q, r'}(a, b, c') \in \overline{\F}$, necessarily with $c' \neq c$.
    We view $T_{p, q, r}(a, b, c)$ and $T_{p, q, r'}(a, b, c')$ as one-point extensions of $\set{a, b}$.
    By the amalgamation property it is possible to define $d(c, c')$ so that $T_{p, q, r}(a, b, c) \cup T_{p, q, r'}(a, b, c') \in \overline{\F}$, but this is impossible since $\set{a, c, c'}$ would form an isosceles triangle as $d(a, c) = q = d(a, c')$.
    Hence, $t(p, q)$ is well-defined and unique.
    
    Next, let $p', q' \in \Dist(\F)$ with $t(p', q') = t(p, q) = r$.
    Hence, without loss of generality, $T_{p, q, r}(a, b, c), T_{p', q', r}(a', b, c) \in \overline{\F}$.
    By the amalgamation property, their union is contained in a space $A$ containing also the triangles $T_{p, p', t(p, p')}(b, a, a')$ and $T_{q, q', t(q, q')}(c, a, a')$.
    Hence, $t(p, p') = d_A(a, a') = t(q, q')$.
    
    Now suppose \ref{itm:coherence} in order to show that it implies \ref{itm:limit}.
    We put $X_{\F} := \Dist(\F)$, and $p + q := t(p, q)$ and $\norm{p} := p$ for $p, q \in X_{\F}$.
    It is enough to show that $+$ is a commutative associative addition with the neutral element $0$, and that $\norm{\cdot}$ is an injective norm.
    By considering degenerate triangles, it is easy to see that $0$ is the neutral element and that $p = -p$ for every $p \in X_{\F}$.
    Clearly, $\norm{\cdot}$ is injective and we have $\norm{p} = 0$ if and only if $p = 0$.
    The triangle inequality for $\norm{\cdot}$ follows from the fact that every $\tuple{p, q, t(p, q)}$ is a triangle.
    The only missing property is the associativity of $+$.
    For every $p, q, r \in \Dist(\F)$ we have $t(p, p + q) = q = t(q + r, r)$, and so by \ref{itm:coherence_b} $t(p, q + r) = t(p + q, r)$.
    
    Finally, suppose \ref{itm:limit} and put $\overline{\F} := \Age(X_\F)$. We show that \ref{itm:class} holds.
    Clearly, $\overline{\F}$ is a hereditary class of isosceles-free spaces extending $\F$ with $\Dist(\overline{\F}) = \Dist(\F)$.
    The amalgamation property is also easy to see and follows from the homogeneity of $X_\F$ as in classical Fraïssé theory~\cite[Theorem~7.1.7]{Hodges}.

    This finishes the proof of the equivalences.
    
    Next we show the uniqueness of $\overline{\F}$.
    On one hand, every $A \in \overline{\F}$ satisfies $d(b, c) = t(d(a, b), d(a, c))$ for every $a, b, c \in A$.
    On the other hand, by induction, every such metric space $A$ is a member of $\overline{\F}$.
    Since $\overline{\F}$ is hereditary and $\Dist(\overline{\F}) = \Dist(\F)$, every at most two-point space with distances from $\Dist(\F)$ is in $\F$.
    For every $A$ of cardinality $\card{A} \geq 3$ we write $A$ as the disjoint union $B \cup \set{x} \cup \set{y}$ for some $x, y \in A$, and we observe that $A$ embeds into any amalgamation of $B \cup \set{x}$ and $B \cup \set{y}$ in $\overline{\F}$.
    This is because in any such amalgamation we have $d(x, y) = t(d_A(b, x), d_A(b, y))$ for any $b \in B$.
    
    The uniqueness of $X_\F$ follows from the fact that for any $0 \in X_\F$ the map $\norm{x} := d(x, 0)$ is a bijection $X_\F \to \Dist(\F)$ and we have $d(x, y) = t(\norm{x}, \norm{y})$ for every $x, y \in X_\F$.
  \end{proof}
\end{proposition}

\begin{remark}
  Note that the previous proposition covers also uncountable distance sets and uncountable limit spaces.
  It is the unique amalagamation that allows us to build uncountable homogeneous structures directly in this case.
\end{remark}

\begin{remark}
  It is known that the class $\F_R$ of all finite metric spaces with distances from a set $0 \in R \subseteq [0, \infty)$ has the amalgamation property if and only if it satisfies the \emph{four-values condition} \cite[Proposition~1.4]{DLPS}, see also \cite[Theorem~1.4]{Sauer}.
  In the context of isoceles-free spaces the analogue of the four-values conditions is the condition \ref{itm:coherence_b}.
  It is formally similar and it also corresponds to the amalgamation of two one-point extensions over a two-point space.
  In fact, when we restrict $\F_R$ to the subclass $\F_{R, t}$ of all isosceles-free spaces compatible with the given scheme $t$ \ref{itm:coherence_a}, then \ref{itm:coherence_b} characterizes the amalgamation property as well.
\end{remark}

\begin{corollary} \label{thm:isosceles-free_amalgamation_class_corollary}
  Let $0 \in R \subseteq [0, \infty)$ and let $t\maps R^2 \to R$ be such that for every $p, q \in R$ we have
  \begin{enumerate}
  \item $t(p, q) = t(q, p) \leq p + q$,
  \item $t(t(p, q), q) = t(p, 0) = p$,
  \end{enumerate}
  Moreover, let $\F_{R, t}$ be the class of all finite metric spaces $A$ such that for every $a, b, c \in A$ we have $t(d(a, b), d(a, c)) = d(b, c)$.
  Then $\F_{R, t}$ is a hereditary class of isosceles-free spaces with $\Dist(\F_{R, t}) = R$, and it has the amalgamation property if and only if $t$ satisfies \ref{thm:isosceles-free_amalgamation_class}~\ref{itm:coherence_b}.
  
  \begin{proof}
    Clearly, $\F_{R, t}$ is a hereditary class of metric spaces.
    Note that we have $t(p, p) = t(t(0, p), p) = 0$ for every $p \in R$,
    and also $t(p, q) \notin \set{p, q}$ if $p, q > 0$ since $t(p, q) = q$ would imply $p = t(t(p, q), q) = t(q, q) = 0$.
    Hence, all members of $\F_{R, t}$ are isosceles-free.
    It is easy to check that the properties of $t$ imply that the triangle $T_{p, q, t(p, q)}(a, b, c)$ is a member of $\F_{R, t}$, and hence $\Dist(\F_{R, t}) = R$ and we can apply Proposition~\ref{thm:isosceles-free_amalgamation_class}.
    Since $\F_{R, t}$ is the largest class compatible with the scheme $t$, we have $\overline{\F_{R, t}} = \F_{R, t}$ if the amalgamation extension exists.
    Hence the claim follows from Proposition~\ref{thm:isosceles-free_amalgamation_class}~\ref{itm:coherence}.
  \end{proof}
\end{corollary}

\begin{example}
  We conclude with an example of a finite set of distances $0 \in R \subseteq [0, \infty)$ with a scheme $t\maps R^2 \to R$ satisfying the conditions of the previous corollary, but not \ref{thm:isosceles-free_amalgamation_class}~\ref{itm:coherence_b}.
  We shall take $R \subseteq \set{0} \cup [1, 2]$ so that the triangle inequality becomes trivial.
  In that case $t$ satisfying the conditions can be equivalently described as a family $\mathcal{T}$ of $3$-point subsets of $R \setminus \set{0}$ such that every $2$-point subset of $R \setminus \set{0}$ is contained in exactly one member of $\mathcal{T}$.
  Let us pick any $9$-point subset $R \setminus \set{0} \subseteq [1, 2]$ and identify it with the $2$-dimensional linear space $\ZZ_3 \times \ZZ_3$ over the $3$-element field.
  Taking $\mathcal{T}$ consisting of the twelve $3$-point affine lines in $\ZZ_3 \times \ZZ_3$ works.
  The condition \ref{thm:isosceles-free_amalgamation_class}~\ref{itm:coherence_b} fails as we have e.g. $t(\tuple{0, 1}, \tuple{0, 2}) = \tuple{0, 0} = t(\tuple{1, 0}, \tuple{2, 0})$, while $t(\tuple{0, 1}, \tuple{1, 0}) = \tuple{2, 2} \neq \tuple{1, 1} = t(\tuple{0, 2}, \tuple{2, 0})$.
\end{example}

\section{Decompositions of homogeneous spaces} \label{sec:decomp}

In this section we study two invariant decompositions of homogeneous spaces based on non-repeating distances.
This will give more insight into the structure of homogeneous spaces that are close to being isosceles-free and will be applied in Section \ref{sec:distances}.

\begin{definition} \label{def:invariant_decomposition}
	We say that an equivalence $\sim$ on a metric space $X$ and the corresponding decomposition $X/{\sim}$ are \emph{invariant} if for every automorphism $f \in \Aut(X)$ and a component $C \in X/{\sim}$ we have that $f[C]$ is a component.
	Equivalently, $x \sim y$ implies $f(x) \sim f(y)$ for every $x, y \in X$.
	
	By $\Aut_*(X)$ we denote the family of all automorphisms $f \in \Aut(X)$ setwise fixing all components, i.e. $f[C] = C$ for every $C \in X/{\sim}$.
\end{definition}

\begin{proposition} \label{thm:invariant}
	Let $\sim$ be an invariant decomposition of a metric space $X$.
	\begin{enumerate}
		\item $\Aut_*(X)$ is a normal subgroup of $\Aut(X)$.
		\item The restriction map $\rho_C\maps \Aut_*(X) \to \Aut(C)$, $f\mapsto f|_C$, is a group homomorphism.
		\item If $X$ is $1$-homogeneous, all components $C \in X/{\sim}$ are isometric.
		\item If $X$ is $n$-homogeneous or ultrahomogeneous, so is every component $C \in X/{\sim}$, and this is witnessed by the subgroup $\im(\rho_C) \leq \Aut(C)$.
	\end{enumerate}
	
	\begin{proof} \hfill
		\begin{enumerate}
			\item Clearly, $\Aut_*(X) \subseteq \Aut(X)$ is a subgroup.
				Also for every $f \in \Aut_*(X)$, $g \in \Aut(X)$, and a component $C$, we have $g[f[g^{-1}[C]] =: g[f[C']] = g[C'] = g[g^{-1}[C]] = C$, and so $\Aut_*(X)$ is a normal subgroup.
			\item This is clear.
			\item For every two components $C, C'$ we pick points $x \in C$ and $y \in C'$.
			Since $X$ is $1$-homogeneous there is $f \in \Aut(X)$ such that $f(x) = y$ and so $f[C] = C'$, showing that $C$ and $C'$ are isometric.
			\item	This follows from Proposition~\ref{thm:ultrahom_subspace} since every component is quasi-invariant.
			\qedhere
		\end{enumerate}
	\end{proof}
\end{proposition}

In the following, we define two invariant decompositions of homogeneous spaces related to isosceles-free spaces.

\begin{definition} \label{def:singleton}
	Let $X$ be a $1$-homogeneous space.
	We say that a distance $s \in \Dist(X)$ is \emph{singleton} if for every $x \in X$ there exists a unique $y \in Y$ such that $d(x, y) = s$,
	i.e. there is no isosceles triangle with side lengths $\tuple{s, s, t}$ in $X$ (for $t > 0$).
	By $1$-homogeneity, it is enough if the defining condition holds at a single point $x \in X$.
	
	Let $S_X \subseteq \Dist(X)$ be the set of all singleton distances in $X$.
	Clearly, $X$ is isosceles-free if and only if $S_X = \Dist(X)$.
\end{definition}

\begin{theorem}[Decomposition into isosceles-free components]\label{thm:isosceles-free-decomposition}
	Let $X$ be a $2$-homogeneous space and let $x \sim y$ if $d(x, y) \in S_X$ for $x, y \in X$.
	We have that $\sim$ is an invariant equivalence relation inducing a decomposition of $X$ into pairwise isometric homogeneous isosceles-free spaces.
\end{theorem}

\begin{proof}
	Symmetry and reflexivity are trivial, transitivity follows from the fact that if we assume $x\sim y\sim z$ and $x\nsim z$, this means that $d(x,z)\notin S$, therefore (due to homogeneity) there exists some point $w\in X$ with $d(x,w)=d(x,z)$, meaning the function $f\colon\{x,z\}\rightarrow\{x,w\}$ which fixes $x$ and maps $z$ to $w$ is an isometry. However, due to $2$-homogeneity, we now know that there exists an automorphism $F$ which extends $f$. It follows that $d(x,y)=d(x,F(y))$, but since we assumed $x\sim y$, this means $y=F(y)$. Analogously, $d(y,z)=d(y,F(z))=d(y,w)$, but because we assumed $y\sim z$, this implies $z=w$.
	
	Clearly, the equivalence $\sim$ is invariant since isometries preserve distances, and the components $C \in X/{\sim}$ are isosceles-free.
	By Proposition~\ref{thm:invariant}, they are also pairwise isometric and $2$-homogeneous, and hence ultrahomogeneous by Proposition~\ref{ultra-free}.
\end{proof}

A simple example of the decomposition into isosceles-free components follows.
It also shows that the map $\rho_C\maps \Aut_*(X) \to \Aut(C)$ is not necessarily injective.
\begin{example}
  Let $C_4 = \set{1, 2, 3, 4}$ be the four element circle graph considered in Example~\ref{circle-graph}.
	Then the isosceles-free components of $C_4$ are $C = \set{1, 3}$ and $C' = \set{2, 4}$, the pairs of antipodal points in $C_4$, with all distances between two components being equal to one.
	
  Now $\Aut_*(C_4) = \set{(1\ 3)^i (2\ 4)^j \in \Sigma_4: i, j \in \set{0, 1}}$ is a $4$-element group, whereas $\Aut(C) = \set{(1\ 3)^i: i \in \set{0, 1}}$ and $\Aut(C') = \set{(2\ 4)^i: i \in \set{0, 1}}$ are both $2$-element groups.
	Hence, the restriction maps $\rho_C, \rho_{C'}$ are surjective, but not injective.
\end{example}

We define another decomposition, which is in a sense dual to the previous one.

\begin{theorem}[Decomposition into isosceles-generated components] \label{thm:iso-generated}
	Let $X$ be a $1$-ho\-mo\-ge\-neous space.
	Let $\sim$ be the equivalence generated by the relation $x \sim y$ if there is $z \neq y$ such that $d(x, y) = d(x, z)$, i.e. we identify points along non-singleton distances, or equivalently collapse all non-degenerate isosceles triangles.
	\begin{enumerate}
		\item
			The equivalence $\sim$ induces an invariant decomposition into isometric $1$-homogeneous components.
			In particular, automorphisms map components onto components.
		\item \label{itm:unique}
			If $\card{X/{\sim}} \geq 2$, then $X$ is uniquely $1$-homogeneous.
		\item \label{itm:swap}
			For every $f \in \Aut(X)$ either all components $C$ are fixed by $f$ (setwise), or none of them are.
			In the latter case we have $f \circ f = \id$.
		\item \label{itm:restr}
			For every $C \in X/{\sim}$ the homomorphism $\rho_C\maps \Aut_*(X) \to \Aut(C)$ is an embedding.
		\item
			For every $h \in \Aut_*(X)$ and $f \in \Aut(X) \setminus \Aut_*(X)$ we have $f \circ h \circ f^{-1} = h^{-1}$.
		\item \label{itm:Boolean}
			If $\card{X/{\sim}} \geq 3$, then $\Aut(X)$ is Boolean, i.e. $X$ is a Boolean metric space.
		\item \label{itm:Abelian}
			If $\card{X/{\sim}} \geq 2$, then $\Aut_*(X)$ is Abelian.
	\end{enumerate}
\end{theorem}
\begin{proof} \hfill
	\begin{enumerate}
		\item
			Clearly, a distance $d(x, y)$ is non-singleton from the point of view of $x$ if and only if it is non-singleton from the point of view of $y$ as the space $X$ is $1$-homogeneous.
			Also then, $d(f(x), f(y))$ is non-singleton for every $f \in \Aut(X)$.
			Hence, the generating symmetric relation and so its reflexive transitive closure is preserved by automorphisms, and we have an invariant decomposition.
			The rest follows from Proposition~\ref{thm:invariant}.
		\item
			Let $f \in \Aut(X)$ such that $f(x) = y$.
			For any $x'$ from a different component than $x$, we have that $d(x, x')$ is a singleton distance, and so $f(x')$ is the unique point $y'$ such that $d(y, y') = d(x, x')$.
			For every $x'' \sim x$ we have that $x''$ is from a different component than $x'$, and we can apply the same argument for $x', x''$.
		\item
			Suppose $y := f(x) \nsim x$ for some $x \in X$.
			Since $d(x, y) = d(y, f(y))$ is a singleton distance, we have $f(y) = x$.
			Hence, $f$ swaps the components $C_x$ and $C_y$ of $x$ and $y$.
			We will show that $f[C] \neq C$ for every component $C$, and $f \circ f = \id$ will follow as we can repeat the above argument for any point $x \in X$.
			
			Let $C_z \ni z$ be a component different from $C_x$ and $C_y$.
			The distances $\tuple{a, b, c} := \tuple{d(x, y), d(y, z), d(z, x)}$ are singleton, and hence pairwise different.
			Moreover, by $1$-homogeneity for every point $w \in X$ and all distinct distances $\set{i, j} \subseteq \set{a, b, c}$ there are unique points $w_i, w_j$ with $d(w, w_i) = i$ and $d(w, w_j) = j$, and necessarily $d(w_i, w_j) = k$ where $\set{i, j, k} = \set{a, b, c}$.
			By applying this to $w = y$, we have $d(y, z) = b$, $d(y, f(z)) = d(f(x), f(z)) = d(x, z) = c$, and so $d(z, f(z)) = a$.
			Hence, the unique automorphism $g \in \Aut(X)$ mapping $z \mapsto x$ also maps $f(z) \mapsto y$.
			It follows that $z \nsim f(z)$ since otherwise we would have $x \sim y$ as $g$ preserves the equivalence $\sim$.
		\item
			The map $\rho_C$ is injective by \ref{itm:unique} if there are at least two components and trivially if there is only one component.
		\item We have $h \circ f \in \Aut(X) \setminus \Aut_*(X)$, and so $(h \circ f)^2 = \id$ by \ref{itm:swap}. On the other hand, $f^2 = \id$ and $(h \circ f)^2 = h \circ f \circ h \circ f = h \circ (f \circ h \circ f^{-1})$.
			Together, this shows $h \circ (f \circ h \circ f^{-1}) = \id$, and so $f \circ h \circ f^{-1} = h^{-1}$.
		\item We already know that $f^2 = \id$ for $f \in \Aut(X) \setminus \Aut_*(X)$.
			Let $h \in \Aut_*(X)$.
			Suppose $C_0, C_1, C_2$ are distinct components of $X$.
			There are $f_1, f_2 \in \Aut(X)$ swapping $C_0$ with $C_1$ and $C_0$ with $C_2$, respectively.
			Hence, $f_2 \circ f_1$ moves $C_1$ onto $C_2$, and so is not a member of $\Aut_*(X)$.
			On one hand, $(f_2 \circ f_1) \circ h \circ (f_2 \circ f_1)^{-1} = h^{-1}$.
			On the other hand, $(f_2 \circ f_1) \circ h \circ (f_2 \circ f_1)^{-1} = f_2 \circ (f_1 \circ h \circ f_1^{-1}) \circ f_2^{-1} = (h^{-1})^{-1} = h$.
			Hence, $h^{-1} = h$ and $h^2 = \id$.
		\item Let $h, k \in \Aut_*(X)$. Since $\card{X/{\sim}} \geq 2$, there is an $f \in \Aut(X) \setminus \Aut_*(X)$.
			We have $h \circ f \notin \Aut_*(X)$.
			On one hand, $(h \circ f) \circ k \circ (h \circ f)^{-1} = k^{-1}$.
			On the other hand, $(h \circ f) \circ k \circ (h \circ f)^{-1} = h \circ (f \circ k \circ f^{-1}) \circ h^{-1} = h \circ k^{-1} \circ h^{-1}$.
			Together, $h \circ k^{-1} \circ h^{-1} = k^{-1}$ and $k \circ h = h \circ k$.
		\qedhere
	\end{enumerate}
\end{proof}

\begin{definition}
	We say that a metric space is \emph{isosceles-generated} if its decomposition into isosceles-generated components has at most one component.
\end{definition}

\begin{example} \label{Dn}
  Given $n \in \NN_+$, we define the space $D_n$ as follows. We take two copies of the cyclic space $C_n$ from Example~\ref{circle-graph} and define the distances in $D_n = C_n \times 2$ between the copies as follows:
  $d(\tuple{i, 0}, \tuple{j, 1}) := q_{j - i}$ for every $i, j \in C_n$, where $q_k$, $k \in \ZZ_n$, are distinct distances not in $\Dist(C_n) = \set{0, \ldots, \floor{n/2}}$ suitable for the triangle inequality, i.e. $0 < \abs{q_i - q_j} \leq 1$ and $q_i > \floor{n/2}$ for every $i, j \in \ZZ_n$.
  We choose e.g. $q_k := \floor{n/2} + 1 + k/n$ for $k \in \ZZ_n$.
	
	The space $D_n$ is $1$-homogeneous:
  We know that $\Aut(C_n)$ consists of rotations $\Phi_k\maps i \mapsto i + k$ and reflections $\Psi_k\maps i \mapsto k - i$ for $k \in \ZZ_n$.
  Let $\sigma\maps 2 \to 2$ denote the unique transposition.
  For every $k \in \ZZ_n$ the maps $\Phi_k \times \id$ and $\Psi_k \times \sigma$ are automorphisms of $D_n$, as the following computations show for every $i, j, k \in \ZZ_n$, $x \in 2$, $f \in \Aut(C_n)$, $g \in \Aut(2)$:
  \begin{align*}
    d\bigl(\tuple{f(i), g(x)}, \tuple{f(j), g(x)}\bigr) 
    &= d_{C_n}(f(i), f(j)) = d_{C_n}(i, j) = d(\tuple{i, x}, \tuple{j, x}), \\
    d\bigl(\tuple{\Phi_k(i), \id(0)}, \tuple{\Phi_k(j), \id(1)}\bigr) 
    &= q_{\Phi_k(j) - \Phi_k(i)} = q_{(k + j) - (k + i)} = q_{j - i} = d(\tuple{i, 0}, \tuple{j, 1}), \\
    d\bigl(\tuple{\Psi_k(i), \sigma(0)}, \tuple{\Psi_k(j), \sigma(1)}\bigr) 
    &= q_{\Psi_k(i) - \Psi_k(j)} = q_{(k - i) - (k - j)} = q_{j - i} = d(\tuple{i, 0}, \tuple{j, 1}).
  \end{align*}
	
	We consider the decomposition into isosceles-generated components.
	Since in $C_n$ every point has the same distance to both its ``neighbours'', the space $C_n$ is isosceles-generated.
	As $D_n$ retains the distances within the copies of $C_n$ and as the distances $q_k$ are singleton, the isosceles-generated components of $D_n$ are the sets $C_n\times \set{0}$ and $C_n\times \set{1}$. Hence $D_n$ consists of two isosceles-generated components, and so is uniquely $1$-homogeneous.
	
	We see that $\Aut_*(D_n) = \set{\Phi_k \times \id: k \in \ZZ_n}$, and so the natural map $\rho\maps \Aut_*(D_n) \to \Aut(C_n)$ is not surjective.
	Also for $n \geq 3$, $\ZZ_n \leq \Aut(D_n)$ is not a Boolean group, and so $D_n$ is not a Boolean metric space.
\end{example}

\begin{remark}
	The two extreme cases of the decomposition of a $1$-homogeneous space into isosceles-generated components are as follows.
	The components are singletons if and only if the space is isosceles-free, and if there is only one component, the space is isosceles-generated.
	Note that for a $2$-homogeneous metric space (so both decompositions make sense) the situation is always extreme, see Theorem~\ref{thm:homogeneous_cases}.	
\end{remark}

Next we consider the question whether for the decomposition into isosceles-generated components the normal subgroup $\Aut_*(X) \leq \Aut(X)$ induces a decomposition of $\Aut(X)$ into a semi-direct product.
In the following we show that it is indeed the case, and moreover we endow the set of components $X/{\sim}$ with a metric turning it into a homogeneous isosceles-free space.

\begin{theorem} \label{thm:isosceles-free-quotient}
	Let $X$ be a $1$-homogeneous space, and let $X/{\sim}$ be its decomposition into isosceles-generated components.
	For $r, q \in \Dist(X)$ let us put $r \sim q$ if there are points $x, x_r, x_q \in X$ such that $d(x, x_r) = r$, $d(x, x_q) = q$, and $x_r \sim x_q$.
	For every $R \in \Dist(X)/{\sim}$ let us pick a distance $q_R \in [0, \infty)$ such that $q_{0/{\sim}} = 0$, $q_R \in [a, 2a]$ for every $R \neq 0/{\sim}$ and fixed $a > 0$, and such that $R \mapsto q_R$ is injective.
	For all components $C, C' \in X/{\sim}$ we put $d(C, C') := q_R$ where $d(x, y) \in R$ for any $x \in C$ and $y \in C'$.
	\begin{enumerate}
		\item \label{itm:dist_eq}
			The relation $\sim$ is a well-defined equivalence on $\Dist(X)$, and for every $x, y \in X$ we have $x \sim y$ if and only if $d(x, y) \sim 0$.
		\item \label{itm:homog_iso-free}
			We can always choose the distances $q_R$ as described, and the induced map $d$ is a well-defined metric on $X/{\sim}$ turning it into a homogeneous isosceles-free space.
	\end{enumerate}
	For every $f \in \Aut(X)$ let $\tilde{f}$ denote the induced bijection on $X/{\sim}$.
	\begin{enumerate}[resume]
		\item \label{itm:comp_auto}
			The map $Q\maps f \mapsto \tilde{f}$ is a surjective homomorphism $\Aut(X) \to \Aut(X/{\sim})$ inducing an isomorphism $\Aut(X)/{\Aut_*(X)} \to \Aut(X/{\sim})$.
	\end{enumerate}
	By a \emph{section} we mean a group homomorphism $S\maps \Aut(X/{\sim}) \to \Aut(X)$  (necessarily an embedding) such that $Q \circ S = \id_{\Aut(X/{\sim})}$, i.e. a section selects a representative $f \in \Aut(X)$ for every $\tilde{f} \in \Aut(X/{\sim})$ in a coherent way.
	\begin{enumerate}[resume]
		\item For every section $S$ we have that $\Aut(X)$ is the semidirect product $\Aut_*(X) \rtimes \im(S)$.
		\item For every $\ZZ_2$-linear base $B \subseteq \Aut(X/{\sim})$ and every map $\beta\maps B \to \Aut(X)$ choosing a representative $\beta(f) \in \Aut(X)$ with $Q(\beta(f)) = f$, there is a unique group embedding $S_\beta\maps \Aut(X/{\sim}) \to \Aut(X)$ extending $\beta$, which is necessarily a section.
		Hence, $\Aut(X) \cong \Aut_*(X) \rtimes \Aut(X/{\sim})$.
	\end{enumerate}
      \end{theorem}
      
\begin{proof} \hfill
	\begin{enumerate}
		\item
			The relation $\sim$ on $\Dist(X)$ is clearly reflexive and symmetric.
			To show transitivity, suppose that a triple $\tuple{x, x_r, x_q}$ witnesses $r \sim q$ and that $\tuple{y, y_q, y_p}$ witnesses $q \sim p$.
			By $1$-homogeneity there is an automorphism $f$ such that $f(x) = y$.
			Since $d(y,f(x_q)) = d(x, x_q) = q = d(y, y_q)$, we have $f(x_r) \sim f(x_q) \sim y_q \sim y_p$, and so $r \sim p$.
			
			Finally, for every $x, y \in X$, if $x \sim y$, then $\tuple{x, x, y}$ witnesses that $0 = d(x, x) \sim d(x, y)$.
			On the other hand, if $d(x, y) \sim 0$, then there are $x' \sim y' \in X$ such that $d(x, y) = d(x', y')$.
			By $1$-homogeneity there is $f \in \Aut(X)$ such that $f(x) = x'$.
			We have $d(x', f(y)) = d(x, y) = d(x', y')$, and so $f(y) \sim y' \sim x' = f(x)$ and $x \sim y$.

		\item
			We can always pick the distances $q_R$ as described since $\card{\Dist(X/{\sim})} \leq \card{\Dist(X)} \leq \card{\RR} = \card{\set{0} \cup [a, 2a]}$ for any $a > 0$.
			
			The map $d$ is well-defined since for every $x, x' \in C$ and $y, y' \in C'$ we have $d(x, y) \sim d(x, y')$ as witnessed by $\tuple{x, y, y'}$ and $d(x, y') \sim d(x', y')$ as witnessed by $\tuple{y', x, x'}$.
			Clearly, $d$ is symmetric, and $d(C, C) = 0$ for every component $C$ since $q_{0/{\sim}} = 0$.
			Also if $d(C, C') = 0$, then for any $x \in C$ and $y \in C'$ we have $d(x, y) \sim 0$, and so $x \sim y$ by \ref{itm:dist_eq} and $C = C'$.
			Finally, $d$ trivially satisfies the triangle inequality since $\Dist(X/{\sim}) \subseteq \set{0} \cup [a, 2a]$.
			
			To show that the space $X/{\sim}$ is isosceles-free, suppose $d(C, C') = d(C, C'') > 0$.
			For any $x \in C$, $x' \in C'$, and $x'' \in C''$ we have $d(x, x') \sim d(x, x'')$, and so there is a witnessing triple $\tuple{y, y', y''}$.
			By $1$-homogeneity there is $f \in \Aut(X)$ such that $f(x) = y$.
			Since the distances $d(x, x')$ and $d(x, x'')$ are singleton, we have $f(x') = y'$ and $f(x'') = y''$, and so $y' \sim y''$ implies $x' \sim x''$ and $C' = C''$.
			
			$X/{\sim}$ is $1$-homogeneous since for all components $C \ni x$ and $C' \ni y$ there is $g \in \Aut(X)$ such that $g(x) = y$, and so $\tilde{g}(C) = C'$, where $\tilde{g}$ is the induced bijection on $X/{\sim}$.
			We have $\tilde{g} \in \Aut(X/{\sim})$ by \ref{itm:comp_auto}.
		\item
			For every $f \in \Aut(X)$ we have $\tilde{f} \in \Aut(X/{\sim})$ since for all points and components $x \in C$ and $y \in C'$ we have
			$d(\tilde{f}(C), \tilde{f}(C')) = q_{d(f(x), f(y))/{\sim}} = q_{d(x, y)/{\sim}} = d(C, C')$.
			Clearly, the assignment $f \mapsto \tilde{f}$ preserves composition and the identity, and so $Q\maps \Aut(X) \to \Aut(X/{\sim})$ is a group homomorphism.
			
			Let $g \in \Aut(X/{\sim})$ and let $x \in C \in X/{\sim}$.
			We take $x' \in g(C)$ and consider $f \in \Aut(X)$ such that $f(x) = x'$.
			We have $\tilde{f}(C) = \tilde{g}(C)$, and hence $\tilde{f} = \tilde{g}$ by Corollary~\ref{unique-automorphism} since $X/{\sim}$ is homogeneous isosceles-free by \ref{itm:homog_iso-free}. 
			Hence, $Q$ is a surjection.
			Finally, $\tilde{f} = \id_{X/{\sim}}$ if and only if $f$ setwise fixes all components, i.e. $f \in \Aut_*(X)$.
			Hence, $Q$ induces an isomorphism $\Aut(X)/{\Aut_*(X)} \to \Aut(X/{\sim})$.
		\item
			This follows from standard group-theoretic facts.
			We already know that $\Aut_*(X) \leq \Aut(X)$ is a normal subgroup.
			We have $\Aut_*(X) \cap \im(S) = \set{\id}$ since every $f \in \Aut(X)$ such that $\tilde{f} \neq \id_{X/{\sim}}$ moves components, and so $f \notin \Aut_*(X)$.
			We have $\Aut_*(X) \circ \im(S) = \Aut(X)$ since for every $f \in \Aut(X)$ and $g := S(\tilde{f}) \in \im(S)$ we have $f = (f \circ g^{-1}) \circ g$ and $f \circ g^{-1} \in \Aut_*(X)$ as $Q(f \circ g^{-1}) = Q(f) \circ Q(g)^{-1} = \tilde{f} \circ \tilde{f}^{-1} = \id$.
		\item
			By Theorem~\ref{thm:homog_iso-free} we know that $\Aut(X/{\sim})$ is indeed a $\ZZ_2$-linear space.
			The subgroup $H \leq \Aut(X)$ generated by $\im(\beta)$ is Boolean since every composition $g := \beta(f_1) \circ \cdots \circ \beta(f_n)$ for $f_i \in \Aut(X/{\sim})$, $i \leq n$, satisfies $g^2 = \id$.
			This is obvious for $n = 0$, and otherwise we have $Q(g) = f_1 \circ \cdots \circ f_n \neq \id$ since $B$ is linearly independent, and so $g \not\in \Aut_*(X)$.
			Hence, we have a map $\beta \maps B \to H$ into a $\ZZ_2$-linear space, which has a unique linear extension $S_\beta\maps \Aut(X/{\sim}) \to H$, which is the same thing as a group homomorphism extension $\Aut(X/{\sim}) \to \Aut(X)$.
			$S_\beta$ is necessarily a section since $Q \circ S_\beta = \id$ holds on $B$ and so on the subgroup generated by $B$.
		\qedhere
	\end{enumerate}
\end{proof}

From the previous theorem and from the structure of finite homogeneous isosceles-free spaces (Theorem~\ref{thm:homog_iso-free}) we obtain the following.
\begin{corollary}
	Let $X$ be a $1$-homogeneous space, and let $X/{\sim}$ be its decomposition into isosceles-generated components.
	We have that the normal subgroup $\Aut_*(X) \leq \Aut(X)$ is complemented, and so $\Aut(X)$ decomposes as a semidirect product $\Aut_*(X) \rtimes (\Aut(X)/\allowbreak\Aut_*(X))$.
	Moreover, if $X$ is finite and nonempty, then $\card{X/{\sim}} = 2^m$ for some $m \in \omega$.
\end{corollary}

In the following we generalize the construction from Example~\ref{Dn} and show that every $1$-homogeneous space with exactly two isosceles-generated components arises this way.
\begin{construction}[Rainbow duplicate] \label{rainbow}
	Let $X$ be a $1$-homogeneous space and let $H \leq \Aut(X)$ be an Abelian subgroup such that for every $x, y \in X$ there is a unique element $h \in H$ (denoted by $h_x^y$) such that $h(x) = y$.
	We define the corresponding \emph{rainbow duplicate} of $X$ as the metric space $X \times_r 2$ with the distance
	\begin{align*}
		d(\tuple{x, 0}, \tuple{y, 0}) &= d(\tuple{x, 1}, \tuple{y, 1}) = d_X(x, y), \\
		d(\tuple{x, 0}, \tuple{y, 1}) &= r(h_x^y),
	\end{align*}
	where $r\maps H \to (0, \infty) \setminus \Dist(X)$ is an injective map such that triangle inequality in $X \times_r 2$ is satisfied.
	We also suppose there exists a map $g \in \Aut(X)$ such that $g \circ g = \id$ and $g \circ h \circ g^{-1} = h^{-1}$ for every $h \in H$.
	
	\begin{enumerate}
		\item \label{itm:rainbow_triangle}
			If $\abs{r(h) - r(h')} \leq \min(\Dist(X) \setminus \set{0})$ and $r(h) \geq \max(\Dist(X))$ for every $h, h' \in H$, then $d$ satisfies the triangle inequality.
		\item We have $\Dist(X \times_r 2) = \Dist(X) \cup \im(r)$, which is a disjoint union.
		\item \label{itm:rainbow_components}
			The decomposition of $X \times_r 2$ into isosceles-generated components refines $\set{X \times \set{0},\linebreak[1] X \times \set{1}}$, and we have equality if and only if $X$ is isosceles-generated.
		\item The map $e_H\maps h \mapsto h \times \id$ is a group embedding $H \to \Aut(X \times_r 2)$.
		\item The map $g \times \sigma$ (where $\sigma\maps 2 \to 2$ is the transposition) generates a copy of $\ZZ_2$ in $\Aut(X \times_r 2)$.
		\item \label{itm:rainbow_auto}
			The rainbow duplicate $X \times_r 2$ is $1$-homogeneous, and the maps above induce an isomorphism $H \rtimes \ZZ_2 \to \Aut(X \times_r 2)$.
	\end{enumerate}
	
	\begin{proof} \hfill
		\begin{enumerate}
			\item The two triangles to check are of the form $\tuple{x, 0}, \tuple{y, 0}, \tuple{z, 1}$ and $\tuple{x, 0}, \tuple{y, 1}, \tuple{z, 1}$ for $x, y, z \in X$, corresponding to the triangles of distances $d_X(x, y), r(h_x^z), r(h_y^z)$ and $r(h_x^y), r(h_x^z), d_X(y, z)$, respectively.
			\item This is clear.
			\item This is also clear.
			\item
				It is enough to show that $f \times \id$ preserves distances for every $f \in H$.
				We have
				\begin{align*}
					d(\tuple{f(x), i}, \tuple{f(y), i}) &= d_X(f(x), f(y)) = d_X(x, y) = d(\tuple{x, i}, \tuple{y, i}), \text{ for $i \in \set{0, 1}$}, \\
					d(\tuple{f(x), 0}, \tuple{f(y), 1}) &= r(h_{f(x)}^{f(y)}) = r(h_x^y) = d(\tuple{x, 0}, \tuple{y, 1}).
				\end{align*}
				The equality $h_{f(x)}^{f(y)} = h_x^y$ for every $x, y \in X$ means $h(f(x)) = f(h(x))$ for every $h \in H$ and $x \in X$ (by substituting $y = h(x)$ and $h = h_x^y$), and it is true since $f \in H$ and $H$ is Abelian.
			\item
				Clearly $(g \times \sigma)^2 = \id$.
				We need to show that $g \times \sigma$ preserves distances.
				We have
				\begin{align*}
					d(\tuple{g(x), \sigma(i)}, \tuple{g(y), \sigma(i)}) &= d_X(g(x), g(y)) = d_X(x, y) = d(\tuple{x, i}, \tuple{y, i}), \text{ $i \in \set{0, 1}$}, \\
					d(\tuple{g(x), 1}, \tuple{g(y), 0}) &= r(h_{g(y)}^{g(x)}) = r(h_x^y) = d(\tuple{x, 0}, \tuple{y, 1}).
				\end{align*}
				The equality $h_{g(y)}^{g(x)} = h_x^y$ for every $x, y \in X$ means $h(g(h(x))) = g(x)$ for every $h \in H$ and $x \in X$, which is equivalent to $g \circ h \circ g^{-1} = h^{-1}$ for every $h \in H$.
			\item
				The maps $h \times \id$ for $h \in H$ witness $1$-homogeneity for pairs of points in the same copy of $X$, while $g \times \sigma$ swaps the copies.
				Hence, the rainbow duplicate is $1$-homogeneous.
				Clearly, $\im(e_H) \cap \set{\id, g \times \sigma} = \set{\id}$ and $(g \times \sigma) \circ (h \times \id) \circ (g \times \sigma)^{-1} = (g \circ h \circ g^{-1}) \times \id = h^{-1} \times \id = (h \times \id)^{-1}$, and so we have an embedding $H \rtimes \ZZ_2 \to \Aut(X \times_r 2)$.
				The embedding is onto since the maps in the image already witness $1$-homogeneity of the rainbow duplicate, and by \ref{itm:rainbow_components} the decomposition into isosceles-generated components has at least two components, and by Theorem~\ref{thm:iso-generated}~\ref{itm:unique} the automorphisms witnessing $1$-homogeneity are unique.
			\qedhere
		\end{enumerate}
	\end{proof}
\end{construction}

\begin{remark}\label{g-is-reflection}
	The spaces $D_n$ from Example~\ref{Dn} are rainbow duplicates: $D_n = C_n \times_r 2$ for $H$ the group of all rotations $\set{\Phi_k: k \in \ZZ_n} \leq \Aut(C_n)$, $g$ any reflection $\Psi_k$, and $r\maps H \to (0, \infty)$ defined by $r(\Phi_k) = q_k$.
\end{remark}

\begin{proposition}
	Let $X \times_r 2$ be a rainbow duplicate for some $r\maps H \to (0, \infty)$.
	\begin{enumerate}
		\item $X \times_r 2$ is uniquely $1$-homogeneous. \label{itm:rainbow_unique}
		\item $X \times_r 2$ is a Boolean metric space if and only if $H$ is a Boolean group. \label{itm:rainbow_Boolean}
		\item An iterated rainbow duplicate $(X \times_r 2) \times_{r'} 2$ for some $r'\maps H' \to (0, \infty)$ can be formed only if the first duplicate $X \times_r 2$ is Boolean.
		In this case, necessarily $H' = \Aut(X \times_r 2)$, and $(X \times_r 2) \times_{r'} 2$ is Boolean as well.
	\end{enumerate}
	
	\begin{proof} \hfill
		\begin{enumerate}
			\item Follows from Construction~\ref{rainbow}~\ref{itm:rainbow_components} and Theorem~\ref{thm:iso-generated}~\ref{itm:unique}.
			\item By Construction~\ref{rainbow}~\ref{itm:rainbow_auto} we have an isomorphism $H \rtimes \ZZ_2 \to \Aut(X \times_r 2)$ where $\ZZ_2$ acts on $H$ by taking the inverse.
				Hence, if $\Aut(X \times_r 2)$ is Boolean, then $H$ is Boolean since it is isomorphic to its subgroup.
				On the other hand, $H$ is Boolean, then the inverse on $H$ is trivial, and so $H \rtimes \ZZ_2$ is a direct product and a Boolean group.
			\item
				Necessarily $H' = \Aut(X \times_r 2)$ since $X \times_r 2$ is uniquely $1$-homogeneous by \ref{itm:rainbow_unique}, and so $\Aut(X \times_r 2)$ is Abelian and $X \times_r 2$ is a Boolean metric space by Corollary~\ref{cor:Boolean}.
				$(X \times_r 2) \times_{r'} 2$ is then Boolean by \ref{itm:rainbow_Boolean}.
			\qedhere
		\end{enumerate}
	\end{proof}
\end{proposition}

\begin{remark}
	It is not hard to see that every finite homogeneous isosceles-free space can be obtained as an iterated rainbow duplicate of a singleton space.
\end{remark}

\begin{remark}\label{admissible-map-exists}
	Let $X$ be a $1$-homogeneous metric space.
	Every Boolean subgroup $H \leq \Aut(X)$ such that for every $x, y \in X$ there is a unique $h \in H$ with $h(x) = y$ is admissible for the rainbow duplicate construction since for any $g, h \in H$ we have $g \circ h \circ g\inv = g \circ g\inv \circ h = h = h\inv$.
	
	Also, if $X$ is finite, then for any admissible subgroup $H$ there is an admissible map $r\maps H \to (0, \infty)$ by Construction~\ref{rainbow}~\ref{itm:rainbow_triangle}.
\end{remark}

\begin{example} \label{ex:boolean_rainbow}
	Let $Y_n$ be a copy of the homogeneous isosceles-free space of size $2^n$ from Example~\ref{ex:power-of-two}, but endowed with the discrete metric $d(x, y) = 1$ for $x \neq y$.
	Then for $H = \Aut(X_n)$ and any injective $r\maps H \to (1, 2)$ we have a Boolean rainbow duplicate $Y_n \times_r 2$ that is not isosceles-free (for $n \geq 2$).
\end{example}

\begin{proposition} \label{thm:two_components}
	Let $Y$ be a $1$-homogeneous metric space with exactly $2$ isosceles-ge\-ne\-ra\-ted components $X, X'$.
	Then $Y$ is isometric to a rainbow duplicate of $X$.
	More precisely, we have the following.
	\begin{enumerate}
		\item \label{itm:q_swap}
			For every $q \in \set{d(x, x'): x \in X, x' \in X'} = \Dist(Y) \setminus \Dist(X)$ there is a unique map $f_q\maps Y \to Y$ such that $d(y, f_q(y)) = q$.
			We have $f_q \circ f_q = \id_Y$, $f_q$ swaps the components $X, X'$, the restriction $f_q\maps X \to X'$ is an isometry, and $f_q \circ g = g \circ f_q$ for every $g \in \Aut(Y)$.
		\item \label{itm:identify}
			The map $\phi\maps \tuple{x, i} \mapsto f_q^i(x)$ is a bijection $X \times 2 \to Y$ such that the metric $d$ on $X \times 2$ turning $\phi$ into an isometry satisfies $d(\tuple{x, i}, \tuple{y, i}) = d_X(x, y)$ for every $x, y \in X$ and $i \in \set{0, 1}$.
		\item \label{itm:H}
			Let $H := \set{h|_X: h \in \Aut_*(Y)} \leq \Aut(X)$. We have that $h \mapsto h \times \id$ is an isomorphism $H \to \Aut_*(X \times 2) \cong_{\phi} \Aut_*(Y)$, $H$ is Abelian, and for every $x, y \in X$ there is unique $h \in H$ with $h(x) = y$.
		\item \label{itm:r}
			For every $h \in H$, the distance $r(h) := d(\tuple{x, 0}, \tuple{h(x), 1})$ does not depend on $x \in X$.
			This defines an injective map $r\maps H \to (0, \infty) \setminus \Dist(X)$.
		\item
			The group $H$ and the map $r$ are admissible parameters for the rainbow duplicate construction, and $\phi\maps X \times_r 2 \to Y$ is an isometric isomorphism.
	\end{enumerate}
	
	\begin{proof} \hfill
		\begin{enumerate}
			\item
				Since $q$ is a distance between the components, it is singleton, and so for every $y \in Y$ there is a unique point $y'$ with $d(y, y') = q$.
				Hence, $f_q$ is well-defined, and $f_q^2 = \id$.
				
				The distance $q$ cannot occur in a single component, i.e. we indeed have $\set{d(x, x'): x \in X, x' \in X'} = \Dist(Y) \setminus \Dist(X)$:
				let $x \in X$ and $x' \in X'$ with $d(x, x') = q$.
				If also $d(y, y') = q$ for some $y, y' \in X$, then the automorphism mapping $x$ to $y$ would have to also map $x'$ to $y'$ since $q$ is a singleton distance, but this is impossible since every automorphism induces a bijection between the components.
				Hence, $f_q$ swaps the components, and the restriction $X \to X'$ is well-defined.
				
				We need to show that $f_q\maps X \to X'$ is an isometry.
				Let $x, y \in X$.
				We have $d(x, f_q(x)) = d(y, f_q(y)) = q$.
				Let $q' := d(x, f_q(y))$.
				Since $x$ and $f_q(y)$ are from different components, $q'$ is a singleton distance.
				By $1$-homogeneity every triangle with side distances $q$ and $q'$ has the same distance $q''$ of the third side.
				We apply this observation to the triangles $\tuple{x, y, f_q(y)}$ and $\tuple{x, f_q(x), f_q(y)}$, and so we have $d(x, y) = q'' = d(f_q(x), f_q(y))$.
				
				We have $d(g(f_q(x)), g(x)) = d(f_q(x), x) = q = d(f_q(g(x)), g(x))$ for every $g \in \Aut(Y)$ and $x \in Y$, and so $g(f_q(x)) = f_q(g(x))$ since $q$ is a singleton distance.
				
			\item 
				The map $\phi$ is is clearly a bijection as it bijectively maps $X \times \set{0}$ to $X$ and $X \times \set{1}$ to $f_q[X] = X'$.
				The rest follows easily from \ref{itm:q_swap}:
				for each $x,y\in X$ and each $i\in\set{0,1}$, we have $d(\tuple{x, i}, \tuple{y, i}) = d_Y(f_q^i(x), f_q^i(y)) = d_X(x, y)$.
			
			\item
				Clearly $h \mapsto h|_X$ is an isomorphism $\Aut_*(Y) \to H$ by Theorem~\ref{thm:iso-generated}~\ref{itm:restr}.
				We just need to observe that $h|_X \times \id$ translates to $h$ via $\phi$, i.e. $\phi \circ (h|_X \times \id) = h \circ \phi$,
				which reduces to $f_q^i(h(x)) = h(f_q^i(x))$ for every $x \in X$ and $i \in \set{0, 1}$.
				But this is true by \ref{itm:q_swap}.
				
				$H$ is Abelian since $\Aut_*(Y)$ is Abelian by Theorem~\ref{thm:iso-generated}~\ref{itm:Abelian}.
				We know that for every $x, y \in X$ there is a unique $k \in \Aut_*(Y)$ with $k(x) = y$, and the restriction $k \mapsto k|_X$ is an isomorphism $\Aut_*(Y) \to H$, so there is a unique $h \in H$ with $h(x) = y$.
			
			\item
				We have $d(\tuple{x, 0}, \tuple{h(x), 1}) = d_Y(f_q^0(x), f_q^1(h(x))) = d_Y(x, f_q(h(x)))$.
				Let $x' \in X$ be another point, and let $k \in \Aut_*(Y)$ be the map such that $k(x) = x'$.
				Then $d_Y(x', f_q(h(x'))) = d_Y(k(x), f_q(h(k(x)))) = d_Y(k(x), k(f_q(h(x)))) = d_Y(x, f_q(h(x)))$ since $H$ is Abelian and $k \circ f_q = f_q \circ k$ by \ref{itm:q_swap}.
				
				We have $r(h) \notin \Dist(X)$ for every $h \in H$ by \ref{itm:q_swap} since $r(h)$ is a distance between the components, and $r$ is injective since for $h \neq h' \in H$, $\tuple{h(x), 1}$ and $\tuple{h'(x), 1}$ are distinct points in the other component.
			
			\item
				$H$ and $r$ are admissible by \ref{itm:H} and \ref{itm:r}.
				We compare the metrics on $X \times 2$ coming from the identification $\phi$ and from the rainbow duplicate construction.
				By \ref{itm:identify} the metrics agree on the components, and by \ref{itm:r} the distances agree between the components.
				
				It remains to show the existence of a suitable witnessing map $g \in \Aut(X)$ for the rainbow duplicate construction.
				Let $f \in \Aut(Y)$ be a map swapping the components and let $g := f_q \circ f|_X \in \Aut(X)$.
				We have $g^2 = (f_q \circ f)^2 |_X = (f_q^2 \circ f^2) |_X = \id_X$, and similarly
				$g \circ h \circ g^{-1} = ((f_q \circ f) \circ \tilde{h} \circ (f_q \circ f)^{-1}) |_X = (f_q^2 \circ f \circ \tilde{h} \circ f^{-1}) |_X = \tilde{h}^{-1}|_X = h^{-1}$ where $\tilde{h} \in \Aut(Y)$ is the extension of $h$.
			\qedhere
		\end{enumerate}
	\end{proof}
\end{proposition}

We finish this section with a classification of $1$-homogeneous metric spaces.
\begin{theorem} \label{thm:homogeneous_cases}
	Let $X$ be a $1$-homogeneous metric space. Then one of the following is true.
	\begin{enumerate}
		\item $X$ is isosceles-generated.
		\item $X$ is a rainbow duplicate of an isosceles-generated space.
		\item $X$ is a Boolean metric space.
	\end{enumerate}
	Suppose that $X$ is even $2$-homogeneous. Then one of the following is true.
	\begin{enumerate}[label=\upshape (\arabic*')]
		\item $X$ is isosceles-generated.
		\item $X$ is isosceles-free.
	\end{enumerate}
	
	\begin{proof}
		By Proposition~\ref{thm:two_components}, if $X$ has exactly two isosceles-generated components, then it is a rainbow duplicate of an isosceles-generated space.
		By Proposition~\ref{thm:iso-generated}~\ref{itm:Boolean}, if $X$ has at least three isosceles-generated components, then it is a Boolean metric space.
		
		If $X$ is $2$-homogeneous, we can consider also the decomposition into isosceles-free components.
		For every two distinct isosceles-generated components $C, C'$ and points $x \in C$ and $x' \in C'$ we have that $d(x, x')$ is a singleton distance, and so $x, x'$ are in the same isosceles-free component. Since $x' \in C'$ was arbitrary, all elements of $C'$ are in the same isosceles-free component as $x$.
		Hence, $C'$ is isosceles-free, and so degenerate as an isosceles-generated component.
		Hence, if there are at least two isosceles-generated components, they are degenerate, and so the whole space is isosceles-free.
	\end{proof}
\end{theorem}

\section{Maximal number of distances} \label{sec:distances}

We have seen that restricting distances of finite metric spaces by a coherent triangle scheme $t\maps R^2 \to R$ (Corollary~\ref{thm:isosceles-free_amalgamation_class_corollary}) leads to an isosceles-free Fraïssé limit $X$.
If the set of distances $R$ is finite, this forces $X$ to be finite as well, despite $X$ being the “largest” and “most complicated” structure associated to the given class of finite structures.
In fact, we get $\card{X} = \card{R}$.
This is in contrast with the situation when distances are restricted just to a given finite set satisfying the four-values condition, where the Fraïssé limit is clearly infinite.

In this section, instead of limiting the set of distances and asking about the cardinality of the Fraïssé limit, we start with a general homogeneous metric space of a given finite cardinality and ask how homogeneity limits the number of attained distances.
Namely, we consider the following question: 
\emph{What is the maximal number of different distances attained in a $k$-homogeneous metric space of cardinality $n \in \NN_+$?}
It turns out that spaces with the highest ratio of the number of attained distances to the number of points are somewhat close to being isosceles-free and that it is useful to view them through the lens of decompositions into isosceles-free and isosceles-generated components, developed in the previous section.

For every metric space $X$ let $\delta(X)$ denote the number of distinct distance values used in $X$, i.e. $\delta(X) := \card{\Dist(X)}$.
Moreover let 
\begin{align*}
	\Delta_k(n) &:= \max\set{\delta(X): \text{$X$ a $k$-homogeneous space with $\card{X} = n$}}, \text{for $k \in \NN_+$}, \\
	\Delta_\omega(n) &:= \max\set{\delta(X): \text{$X$ an ultrahomogeneous space with $\card{X} = n$}}.
\end{align*}
Clearly, we have $\Delta_\omega(n) \leq \cdots \leq \Delta_2(n) \leq \Delta_1(n)$ for every $n \in \NN_+$.

\begin{example} \label{circle-graph-2}
	The circle graph space~$C_n$ considered in Example~\ref{circle-graph} is ultrahomogeneous and $\delta(C_n) = \floor{\frac{n}{2}} + 1$.
	Hence, $\Delta_\omega(n) \geq \floor{\frac{n}{2}} + 1$.
\end{example}

Recall that by $S_X:=\{r\in[0,\infty)\colon \forall x\in X\, \exists!y\in X\colon d(x,y)=r\}\subseteq \Dist(X)$ we mean the set of singleton distances from Definition \ref{def:singleton}.

\begin{observation} \label{singleton-bound}
	We have $\delta(X) \leq (\card{S_X} + \card{X}) / 2$ for every finite $1$-homogeneous space $X$.
	This is because for any $x \in X$ the map $d(x, \cdot)\maps X \to \Dist(X)$ is a surjection where exactly the elements of $S_X$ have a unique preimage.
	Hence, $\card{X} \geq \card{S_X} + 2 (\delta(X) - \card{S_X})$.
\end{observation}

\begin{proposition} \label{n-distances}
	Let $X$ be a finite $1$-homogeneous $n$-point space.
	Clearly, $\delta(X) \leq n$.
	We have $\delta(X) = n$ if and only if $X$ is isosceles-free, and in this case $X$ is ultrahomogeneous and $n$ is a power of two.
	
	In other words, $\Delta_\omega(n) = \Delta_1(n) = n$ if $n = 2^m$, and $\Delta_1(n) < n$ otherwise.
	
	\begin{proof}
		By Observation~\ref{thm:evaluation}, for every $a \in X$ the map $D_a\maps x \mapsto d(a, x)$, $X \to \Dist(X)$, is surjective, and $X$ is isosceles-free if and only if $D_a$ is injective, which happens if and only if $\delta(X) = n$.
		In this case, $X$ is ultrahomogeneous by Proposition~\ref{ultra-free}, and $n$ is a power of two by Corollary~\ref{thm:power_two}.
		A homogeneous isosceles-free space of cardinality $2^m$ exists by Example~\ref{ex:power-of-two}.
	\end{proof}
\end{proposition}

\begin{theorem} \label{thm:number-of-distances}
	For a finite $2$-homogeneous metric space $X$ of $n := \card{X}$ elements, where $n = 2^m (2k + 1)$, we have $\delta(X) \leq 2^m(k+1) =: \beta_n$.
\end{theorem}

\begin{proof}
	Since the space $X$ is $2$-homogeneous, we may consider its decomposition into isosceles-free components $X/{\sim}$ (Theorem~\ref{thm:isosceles-free-decomposition}).
	We have $\card{S_X} = \card{C}$ for any $C \in X/{\sim}$.
	Moreover, $\card{C}$ is a power of two since $C$ is a homogenous isosceles-free space, and $\card{C} \leq 2^m$ since $X/{\sim}$ is a decomposition into pairwise-isometric subspaces and so $\card{C}$ is a factor $\card{X}$.
	Together, by Observation~\ref{singleton-bound} we have
	\[
		\delta(X) \leq \frac{\card{S_X} + \card{X}}{2} \leq \frac{2^m + 2^m(2k + 1)}{2} = 2^m(k + 1). \qedhere
	\]
\end{proof}

The next example witnesses that this bound is optimal.

\begin{example} \label{bound_attained}
	Let $B_{m, k}:=2^mC_{2k+1} \times_1 \tuple{2^m,\norm{\cdot}}$, the $\ell_1$-product of a scaled-up circle graph $C_{2k+1}$ from Example \ref{circle-graph} with the isosceles-free space $\tuple{2^m,\norm{\cdot}}$ from Example \ref{ex:power-of-two}. The space $B_{m,k}$ is ultrahomogeneous, has $n=2^{m}(2k+1)$ elements and satisfies $\delta(B_{m,k})=\beta_n$.
\end{example}

\begin{proof}
	We already know that $C_{2k+1}$ is ultrahomogeneous with $\floor{\frac{2k+1}{2}} + 1 = k + 1$ distances from Example~\ref{circle-graph} and that $\tuple{2^m, \norm{\cdot}}$ is ultrahomogeneous with $2^m$ distances from Example~\ref{ex:power-of-two}.
	Multiplying the usual metric on $C_{2k+1}$ by $2^m$ to get the metric space $2^mC_{2k+1}:=\tuple{C_{2k+1},2^md_{C_{2k+1}}}$ makes sure that $+\maps \Dist(2^m C_{2k+1}) \times \Dist(\tuple{2^m, \norm{\cdot}}) \to \Dist(B_{m, k})$ is injective, and we may use Proposition~\ref{thm:homogeneous_product} to conclude that $B_{m, k}$ is ultrahomogeneous with $2^m(k+1)$ distinct distances.
\end{proof}

\begin{corollary}
	For every $n \in \NN_+$ we have $\Delta_\omega(n) = \Delta_2(n) = \beta_n \leq \Delta_1(n) \leq n$.
\end{corollary}

Now let us bound the number of distances used in $1$-homogeneous spaces.

\begin{proposition}
	Let $X$ be a $1$-homogeneous metric space $X$ of $n := \card{X}$ elements.
	If $n = 2k + 1$, then $\delta(X) \leq k + 1$.
	Hence, $\Delta_1(n) = \Delta_2(n) = \beta_n$ for every odd $n$.
\end{proposition}

\begin{proof}
	Note that for any distance $r\in\Dist(X)$ which is singleton, we can find pairs $\set{x, y_r(x)}$ where $y_r(x)$ is the unique element of $X$ satisfying $d(x,y_r(x))=r$ (note this is symmetrical since $x=y_r(y_r(x))$).
	Clearly, if $r>0$, this would lead to a decomposition of $X$ into pairs of elements, which is impossible since $\card{X}$ is odd.
	Hence, $S_X = \set{0}$, and we have $\delta(X) \leq (\card{S_X} + \card{X})/2 = (1 + (2k + 1))/2 = k + 1$ by Observation~\ref{singleton-bound}.
\end{proof}

\begin{example}\label{ex:circle-rainbow}
  In Example~\ref{Dn} for $n \in \NN_+$ we constructed the $1$-homogeneous space $D_n = C_n\times_r 2$ with
  \[
    \card{D_n} = 2n = 
    \begin{cases}
      4\ell, \\
      4\ell + 2,
    \end{cases}
    \text{and}\quad
    \delta(D_n) = \floor{n/2} + 1 + n = 
    \begin{cases}
      3\ell + 1, & n = 2\ell, \\
      3\ell + 2, & n = 2\ell + 1.
    \end{cases}
  \]
  We have $\delta(D_n) > \beta_{2n}$ for every $n$ that is not a power of two, i.e. we break the optimal bound for $2$-homogeneous spaces.
  
  The construction of $D_n$ can be compared to Example~\ref{bound_attained} in the case of two cycles of odd length.
  There we retain more symmetry while losing more distances, while here we lose more symmetry while keeping more distances.
\end{example}
\begin{proof}
  We observe that $\delta(D_n) > \beta_{2n}$ for every $n = 2^m(2k + 1)$ that is not a power of two, i.e. $k > 0$.
  If $n$ is odd, we have $n = 2k + 1$ and $\delta(D_n) = 3k + 2 > 2k + 2 = \beta_{2n}$.
  If $n$ is even, we have $n = 2\ell$ for $\ell = 2^{m - 1} (2k + 1)$. Hence, 
  \begin{align*}
    \delta(D_n) &= 3\ell + 1 = 2^{m - 1} (6k + 3) + 1, \\
    \beta_{2n} &= 2^{m + 1} (k + 1) = 2^{m - 1} (4k + 4),
  \end{align*}
  and again $\delta(D_n) > \beta_{2n}$ since $k \geq 1$.
\end{proof}

\begin{corollary}
	We have $\Delta_2(n) < \Delta_1(n) < n$ for even $n$ that is not a power of two.
\end{corollary}

To retain more distances in a space of size $2n$, we can improve the previous example by taking the space $B_{m, k}$ instead of $C_n$ as a base for the rainbow duplicate construction.

\begin{example}\label{ex:rainbow-on-2-homogeneous-space}
	Let $E_{m,k} := B_{m,k} \times_r 2$ be a rainbow duplicate of $B_{m,k}$.
	If we put $n := 2^m(2k + 1)$, we have $\card{E_{m, k}} = 2n$ and $\delta(E_{m, k}) = 2^m (3k + 2) =: \alpha_{2n} \geq \beta_{2n} = 2^m(2k + 2)$.
	Moreover, we have $\alpha_{2n} \geq \delta(D_n)$ and even $\alpha_{2n} > \delta(D_n)$ if $m \geq 2$.
	
	\begin{proof}
		Given we can indeed form a rainbow duplicate $B_{m, k} \times_r 2$, which we show in a moment, we have
		\[
			\delta(E_{m, k}) = \delta(B_{m, k}) + \card{\im(r)} = \beta_n + n = 2^m(k + 1) + 2^m(2k + 1) = 2^m(3k + 2).
		\]
		Since $\delta(D_n) = \delta(C_n) + n$, the comparison between $\alpha_{2n}$ and $\delta(D_n)$ reduces to comparison between $\beta_n$ and $\floor{n/2} + 1$.
		For $n$ odd (i.e. $m = 0$) we have $\beta_n = k + 1 = \floor{n/2} + 1$.
		For $n$ even we have $\beta_n = 2^{m - 1}(2k + 2) \geq 2^{m - 1}(2k + 1) + 1 = \floor{n/2} + 1$, which becomes a strict inequality if $m \geq 2$.
		
		To form a rainbow duplicate $B_{m, k} \times_r 2$, we need:
		\begin{enumerate}
			\item\label{itm:H-exists}
				an Abelian subgroup $H\leq \Aut(B_{m,k})$ such that for each $x,y\in X$, there exists precisely one $h\in H$ with $h(x)=y$ (denoted by $h_x^y$),
			\item\label{itm:g-exists}
				a map $g\in\Aut(B_{m,k})$ such that $g^2=\id$ and $g\circ h\circ g^{-1}=h^{-1}$ for every $h\in H$,
			\item
				an admissible map $r\maps H \to (0, \infty)$, but such map exists by Remark~\ref{admissible-map-exists}.
		\end{enumerate}
		
		So let us start with \ref{itm:H-exists}.
		By definition, $B_{m,k} = (2^mC_{2k+1}) \times_1 X_m$, so $\tuple{\phi, \psi} \mapsto \phi \times \psi$ is a group isomorphism $\Aut(C_{2k+1}) \times \Aut(X_m) \to \Aut(B_{m,k})$ by Proposition~\ref{thm:homogeneous_product}.
		As we have shown in Example~\ref{Dn}, if we restrict ourselves to rotations in $C_{2k+1}$, we get an Abelian subgroup $H_1\leq \Aut(C_{2k+1})$ which has the desired property on $C_{2k+1}$ (and any of its re-scalings), whereas Corollary~\ref{unique-automorphism} tells us that $H_2:=\Aut(X_m)$ already has it on $X_m$, and $H_2$ is Abelian since it is Boolean.
		It is a standard fact of direct group products that this implies $H_1\times H_2=:H\leq \Aut(B_{m,k})$, while uniqueness of $h_x^y$ follows from the uniqueness of the representation of an element $h\in H$ in the direct product as a tuple $\tuple{h_1,h_2}$ as well as the uniqueness of $(h_1)_x^y$ and $(h_2)_{x'}^{y'}$ in the two factors. And of course, $H$ is again Abelian.
		
		What remains to be shown is \ref{itm:g-exists}. However, we have shown in \ref{g-is-reflection} already that we can choose any reflection $\Psi_i$ on $C_{2k+1}$ as a function $g_1$, say $g_1:=\Psi_0$, and since $X_m$ is Boolean, we can choose $g_2:=\id_{X_m}$. From this we get that $\tuple{g_1,g_2}$ will naturally satisfy \ref{itm:g-exists} since a direct product's group operation distributes into the components.
	\end{proof}
\end{example}

\begin{proposition}\label{2timesPrime}
	The number of distances in a $1$-homogeneous space $X$ of cardinality $n=2(2k+1)$ with $2k+1$ prime is bounded from above by $3k+2$.
	Hence, $\Delta_1(n) = \alpha_n$ for such $n$.
\end{proposition}

\begin{proof}
	Assume $\delta(X)\geq 3k+2$. Since $\card X=4k+2$, this means there are at least $2k+2$ singleton distances in $X$. Therefore it is possible to pick two non-zero singleton distances $s_1>s_2>0$ and define $f_1$ as the function which maps each $x\in X$ to the unique point of distance $s_1$ from $x$, and analogously define $f_2(x)$ by $d(f_2(x),x)=s_2$ for all $x$.
	
	Now, define a sequence starting at an arbitrary point $x_0\in X$ as:
	\begin{align*}
		x_{2i+1}&:=f_1(x_{2i}),&x_{2i+2}&:=f_2(x_{2i+1}),&i&\in\omega.
	\end{align*}
	Clearly, for each $x_i$, $f_j(x_i)\neq x_i$ since they have distance $s_j>0$ from each other, and for $\ell\neq j$ we have $f_\ell(f_j(x_i))\neq x_i$ since their distances from $f_j(x_i)$ are distinct.
	
	So we have a non-trivial sequence $\tuple{x_n}_{n\in\omega}$ in the finite space $X$ and we want to find the smallest $m\in\NN^+$ such that $x_m=x_i$ for some $i<m$. Note that since $s_1$ and $s_2$ are singleton distances and $d(x_{2i},x_{2i+1})=s_1$ and $d(x_{2i+2},x_{2i+1})=s_2$ for all $i$, the first $m$ points $x_i$ form a `chain' where each point in said chain except $x_0$ and $x_m$ are already connected to their unique partners of distance $s_1$ and $s_2$, respectively. So the only possible way to achieve $x_m=x_i$ for some $i<m$ is if $i=0$ and thus the sequence is $m$-periodic.
	
	Moreover, changing the starting point to $x_0'\not\in \set{x_n\colon n\in\omega}$ will yield a disjoint set $\set{x_n'\colon n\in\omega}$, again because no point $x_i$ has more than one point $y$ of distance $s_j$ to $x_i$ in $X$. Due to $1$-homogeneity, it follows that $\card{\set{x_n\colon n\in\omega}}=m \mid 2(2k+1)=\card X$. Since we have already shown that $m>2$, it follows that $m=2k+1$ or $m=4k+2$, however, if $m$ were $2k+1$ and therefore odd, this would imply that $x_0=x_m=x_{2k+1}=f_1(x_{2k})$, meaning $x_1=x_{m-1}$ since they both have distance $s_1$ to $x_0$, contradicting the minimality of $m$. It follows that $m=4k+2=\card X$ and thus $\set{x_n: n\in\omega}=X$.
	
	We can use this observation to find $k$ distances which are repeated in $X$. To do this, first notice that each automorphism $g\in\Aut(X)$ acts on $X$ like a rotation or reflection, i.e. there always exists an $i\in\omega$ such that either $g(x_j)=x_{i+j}$ for all $j$ (if $i$ is even), or $g(x_j)=x_{i-j}$ (if $i$ is odd). This is because $g$ preserves distances, so it, in particular, preserves pairs $\set{x,f_j(x)}$ of distance $s_j$. This means $g$ commutes with both $f_1$ and $f_2$.
	
	It follows that if $g$ maps $x_0$ to $x_i$, then
	\begin{align*}
		g(x_j)=g(f_{\ell_j}\circ\dotsc\circ f_1(x_0))=f_{\ell_j}\circ\dotsb\circ f_1(x_i)=
		\begin{cases}
			x_{i+j},&i\text{ even,}\\
			x_{i-j},&i\text{ odd.}
		\end{cases}
	\end{align*}
	This is because for even $i$, $f_1(x_i)=x_{i+1}$ whereas, for odd $i$, it is $x_{i-1}$, and from that point onwards $f_1$ and $f_2$ alternate.
	
	Let $g_i\in\Aut(X)$ be the automorphism satisfying $g_i(x_0)=x_i$. Then, for each even $i$, notice that
	$d(x_{m-i},x_0)=d(g_i(x_{m-i}),g_i(x_0))=d(x_0,x_i)=:r_i$, and the function $d(\cdot,x_0)$ therefore assumes at most $k+1$ distinct values on the $2k+1$ points in $X$ with an even index $\set{x_{2n}\colon n\in\omega}$, and of course at most $2k+1$ distinct values on the remaining $2k+1$ odd-indexed points, yielding an upper bound of $\delta(X)\leqslant3k+2 = \alpha_n$.
\end{proof}

\begin{proposition}
	We have $\Delta_1(n) \leq n - 2$ for every $n \geq 7$ that is not a power of two.
	
	\begin{proof}
		Let $X$ be a $1$-homogeneous space of cardinality $n$.
		If $n$ is not a power of two, then $\delta(X) \leq n - 1$ by Proposition~\ref{n-distances}.
		Suppose $\delta(X) = n - 1$.
		That means there is exactly one $r \in \Dist(X)$ that is not a singleton distance, and for every $x \in X$ there are exactly two points $y \in X$ with $d(x, y) = r$.
		
		Let $x_1, x_2 \in X$ be such that $d(x_1, x_2) = r$.
		Then for every $i \geq 2$ there is a unique point $x_{i + 1} \in X$ such that $d(x_i, x_{i + 1}) = r$ and $x_{i + 1} \neq x_{i - 1}$.
		Since $X$ is finite, there is the smallest $k$ such that $x_{k + 1} \in \set{x_1, \ldots, x_k}$.
		Necessarily, $k \geq 3$ (otherwise $r$ would be singleton) and $x_{k + 1} = x_1$ (if $x_{k + 1} = x_i$ with $1 < i \leq k$, then $i < k - 1$ and $x_{i - 1}, x_{i + 1}, x_k$ are three distinct points with distance $r$ from $x_i$).
		Hence we have obtained a cycle $C := \set{x_i: i \in \ZZ_k}$ such that $d(x_i, x_{i + 1}) = r$ for every $i \in \ZZ_k$.
		
		Let $q := d(x_1, x_{-1})$.
		By $1$-homogeneity, for every $i \in \ZZ_k$ there is $f \in \Aut(X)$ such that $f(x_0) = x_i$.
		It follows that $\set{x_{i + 1}, x_{i - 1}} = \set{y \in X: d(y, x_i) = r} = \set{f(y) \in X: d(y, x_0) = r} = \set{f(x_1), f(x_{-1})}$, and so $d(x_{i + 1}, x_{i - 1}) = q$.
		Hence, $d(x_i, x_{i + 2}) = d(x_i, x_{i - 2}) = q$ for every $i \in \ZZ_k$.
		Since $r$ is the only non-singleton distance, we have either $q \neq r$, $x_{i + 2} = x_{i - 2}$, and $k = 4$, or $q = r$, $x_{i + 2} = x_{i - 1}$, and $k = 3$.
		
		Since $r$ is the only non-singleton distance, $C$ is a component of the decomposition $X/{\sim}$ into isosceles-generated components.
		Since $\card{X}$ is not a power of two, $X$ is not a Boolean space, and so $\card{X/{\sim}} \in \set{1, 2}$.
		Hence $\card{X} = \card{C} \cdot \card{X/{\sim}} \in \set{3, 4, 6, 8}$.
	\end{proof}
\end{proposition}

Table~\ref{table} summarizes the maximal number of distances in small homogeneous spaces.
We use the following results obtained throughout this section.
Namely, we use the bounds $\beta_n$ from Theorem~\ref{thm:number-of-distances} and $\alpha_n$ from Example~\ref{ex:rainbow-on-2-homogeneous-space} (where we define define $\alpha_n := \beta_n$ for $n$ odd).
\begin{itemize}
	\item For every $n \in \NN_+$ we have $\Delta_\omega(n) = \Delta_2(n) = \beta_n \leq \alpha_n \leq \Delta_1(n) \leq n$.
	\item For $n$ power of two or odd we have $\Delta_1(n) = \Delta_2(n) = \beta_n = \alpha_n$.
	\item For $n = 2(2k + 1)$ with $2k + 1$ prime we have $\Delta_1(n) = 3k + 2 = \alpha_n$.
	\item For $7 \leq n$ not power of two we have $\Delta_1(n) \leq n - 2$.
\end{itemize}

\begin{table}[ht!] 
	\centering
	\begin{tabular}{cccl}
		$n$ & $\Delta_2(n)$ & $\Delta_1(n)$ & argument for $\Delta_1$ upper bound \\
		\hline
		$1$ & $1$ & $1$ & power of two \\
		$2$ & $2$ & $2$ & power of two \\
		$3$ & $2$ & $2$ & odd \\
		$4$ & $4$ & $4$ & power of two \\
		$5$ & $3$ & $3$ & odd \\
		$6$ & $4$ & $\mathbf{5}$ & two times odd prime\\
		$7$ & $4$ & $4$ & odd \\
		$8$ & $8$ & $8$ & power of two \\
		$9$ & $5$ & $5$ & odd \\
		$10$ & $6$ & $\mathbf{8}$ & two times odd prime\\
		$11$ & $6$ & $6$ & odd \\
		$12$ & $8$ & $\mathbf{10}$ & $\Delta_1(n) \leq n - 2$ \\
		$13$ & $7$ & $7$ & odd \\
		$14$ & $8$ & $\mathbf{11}$ & two times odd prime\\
		$15$ & $8$ & $8$ & odd \\
		$16$ & $16$ & $16$ & power of two \\
		$17$ & $9$ & $9$ & odd \\
		$18$ & $10$ & $\mathbf{\geq 14, \leq 16}$ & $\Delta_1(n) \leq n - 2$ \\
		$19$ & $10$ & $10$ & odd \\
		$20$ & $12$ & $\mathbf{\geq 16, \leq 18}$ & $\Delta_1(n) \leq n - 2$ \\
	\end{tabular}
	
	\caption{Maximal number of distances in small spaces.}
	\label{table}
\end{table}

\begin{question}
	What are the values of $\Delta_1(n)$ for the cases not covered so far ($n = 2^m(2k + 1)$ for $k > 0$ where $m \geq 2$ or $m = 1$ and $2k + 1$ non-prime)?
\end{question}

\begin{remark}
	Note that when searching for $\Delta_1(n)$ where $n$ is not a power of two, a space $X$ of cardinality $n$ cannot be Boolean, and hence by Theorem~\ref{thm:homogeneous_cases} is either isosceles-generated, or a rainbow duplicate of an isosceles-generated space $Y$, and in the latter case $\delta(X) = \delta(Y) + n/2$.
	Hence, $\Delta_1(n)$ ultimately depends on numbers of distances of isosceles-generated spaces.
\end{remark}

\begin{remark}
	Let us note that the question of maximal number of distances in a finite homogeneous metric space can be rephrased in the language of colorings of complete graphs.
	Let $X$ be a set and let $c\maps X \times X \to C_X \ni 0$ be a surjective map onto a set $C_X$ such that
	\begin{itemize}
		\item $c(x, y) = 0$ if and only if $x = y$, i.e. $0$ is a special color reserved for the diagonal,
		\item $c(x, y) = c(y, x)$, i.e. $c$ is symmetric.
	\end{itemize}
	The map $c$ can be naturally viewed as a (not necessarily proper) edge coloring of the complete graph on $X$.
	Let us call pairs $\tuple{X, c}$ \emph{(edge-)colored complete graphs}.
	
	Note that for every metric space $X$, the distance $d\maps X \times X \to \Dist(X)$ is a valid coloring, and that the notions of automorphisms and $n$-homogeneity are exactly the same when we view $X$ as a colored graph instead of a metric space.
	$1$-homogeneity of the metric would be more traditionally called \emph{vertex-transitivity} of the induced coloring.
	Also note that for every finite colored complete graph $X$ we may consider an embedding $e\maps C_X \to \set{0} \cup [a, 2a] \subseteq \RR$ such that $e(0) = 0$ for some $a > 0$, and put $d(x, y) = e(c(x, y))$. This defines a metric on $X$ inducing an equivalent coloring.
	Since we take positive distances in $[a, 2a]$, the triangle inequality becomes trivial.
	This is what we have done in Theorem~\ref{thm:isosceles-free-quotient}.
	
	Altogether, the question of maximal number of distances can be reformulated as: 
	``What is the maximal number of colors used by a vertex-transitive coloring of a complete graph of cardinality $n$?''
\end{remark}

\paragraph{Acknowledgements.} The research of A.~Bartoš and W.~Kubiś was supported by GA ČR (Czech Science Foundation) grant EXPRO 20-31529X and by the Czech Academy of Sciences (RVO 67985840). The research of C.~Bargetz and F.~Luggin was supported by the Austrian Science Fund (FWF): I~4570-N.

\medskip
\noindent
This version of the article has been accepted for publication, after peer review, but is not the Version of Record and does not reflect post-acceptance improvements, or any corrections.
The Version of Record is available online at: \url{http://dx.doi.org/10.1007/s13398-024-01587-y}.

\paragraph{Conflicts of interest.} All authors declare that they have no conflicts of interest.

\bibliographystyle{siam}
\bibliography{references}

\end{document}